\documentclass[11pt]{article}

\usepackage[T1]{fontenc}
\usepackage[margin=1in]{geometry}
\usepackage{microtype}
\usepackage{amsmath,amssymb,amsthm,mathtools}
\usepackage{enumitem}
\usepackage{array,booktabs,tabularx}
\usepackage{authblk}
\usepackage{float}
\usepackage{needspace}
\usepackage{xcolor}
\usepackage{tikz}
\usepackage{url}
\usepackage[hypertexnames=false,colorlinks=true,linkcolor=blue!55!black,
  citecolor=blue!55!black,urlcolor=blue!55!black]{hyperref}

\numberwithin{equation}{section}

\hypersetup{
  pdftitle={The Erdos-Hajnal Property for the six-vertex Graph
    with Edge Set \{ab,bc,cd,de,af,bf,df\}},
  pdfauthor={Viet-Hoang Tran; Tan M. Nguyen}
}

\theoremstyle{plain}
\newtheorem{theorem}{Theorem}[section]
\newtheorem{proposition}[theorem]{Proposition}
\newtheorem{lemma}[theorem]{Lemma}
\newtheorem{corollary}[theorem]{Corollary}

\theoremstyle{definition}
\newtheorem{definition}[theorem]{Definition}

\title{The Erd\H{o}s--Hajnal Property for the six-vertex Graph\\
with Edge Set $\{ab,bc,cd,de,af,bf,df\}$}
\author[1]{Viet-Hoang Tran}
\author[1]{Tan M. Nguyen}
\affil[1]{Department of Mathematics, National University of Singapore}
\date{}

\begin{document}

\maketitle

\begin{abstract}
We prove that the six-vertex graph with edge set $\{ab,bc,cd,de,af,bf,df\}$ has the  Erd\H{o}s--Hajnal property. The proof adapts the iterative-sparsification method of Nguyen, Scott, and Seymour within the comb-based framework of Huang, Ju, and Zhou.
\end{abstract}

\begingroup
\footnotesize
\tableofcontents
\endgroup
\clearpage

\section{Introduction}
\label{sec:introduction}

\begin{table}[!t]
\centering
\footnotesize
\setlength{\tabcolsep}{4pt}
\renewcommand{\arraystretch}{1.05}
\caption{Erd\H{o}s--Hajnal status of $N,J,E,\mathrm{Bird},F$.}
\vspace{6pt}
\label{tab:six-vertex-status}
\begin{tabularx}{\textwidth}{@{}
 >{\raggedright\arraybackslash}p{0.09\textwidth}
 >{\raggedright\arraybackslash}X@{}}
\toprule
\textbf{Graph} & \textbf{Erd\H{o}s--Hajnal status and source} \\
\midrule
$N$
& Erd\H{o}s--Hajnal; proved by Nguyen--Scott--Seymour
  \cite[Theorem~1.4 and Figure~1]{NguyenScottSeymourDensityIV}. \\
\addlinespace[3pt]
$J$
& Erd\H{o}s--Hajnal; follows by combining the Nguyen--Scott--Seymour viral leaf/co-leaf
  result~\cite[Theorem~7.8 and p.~22]{NguyenScottSeymourDensityIV} and
  $P_5$ theorem~\cite{NguyenScottSeymourP5} with the Buci\'c--Fox--Pham
  Erd\H{o}s--Hajnal/viral equivalence
  \cite[Theorem~4]{BucicFoxPham2024}. \\
\addlinespace[3pt]
$E$
& Erd\H{o}s--Hajnal; proved by Huang--Ju--Zhou in a preprint
  \cite[Theorem~1.10]{HuangJuZhou2026}. \\
\addlinespace[3pt]
$\mathrm{Bird}$
& Erd\H{o}s--Hajnal; proved by Huang--Ju--Zhou in the same preprint
  \cite[Theorem~1.11]{HuangJuZhou2026}. \\
\addlinespace[3pt]
$F$
& Erd\H{o}s--Hajnal; proved here.  The case was open before this paper,
  to our knowledge; we give a new $F$-specific comb analysis. \\
\bottomrule
\end{tabularx}
\end{table}

For a graph $G$, let $\alpha(G)$ and $\omega(G)$ denote its stability
number and clique number.  Erd\H{o}s and Hajnal conjectured that for every
graph $H$ there is a constant $c=c(H)>0$ such that every $H$-free graph
$G$ satisfies
\begin{equation}
             \max\{\alpha(G),\omega(G)\}\geq |G|^c.
             \label{eq:EH-conjecture}
\end{equation}
Here $H$-free means that no induced subgraph of $G$ is isomorphic to
$H$.  When \eqref{eq:EH-conjecture} holds, $H$ is
\emph{Erd\H{o}s--Hajnal}; equivalently, it has the Erd\H{o}s--Hajnal
property.  The conjecture predicts a
polynomial-size clique or stable set, in sharp contrast with the
logarithmic bound furnished
by Ramsey's theorem for unrestricted graphs.  Erd\H{o}s and Hajnal
proved, for every fixed $H$ and all sufficiently large $|G|$, the
general lower bound
\[
 \max\{\alpha(G),\omega(G)\}
      \geq 2^{c_H\sqrt{\log |G|}},
\]
and Buci\'c, Nguyen, Scott, and Seymour improved it to
\[
 \max\{\alpha(G),\omega(G)\}
      \geq
      2^{c_H\sqrt{\log |G|\log\log |G|}}.
\]
The conjectured polynomial bound remains open for general $H$;
see~\cite{BucicNguyenScottSeymour2024,ErdosHajnal1989}.

The property is invariant under complementation.  A second fundamental
closure theorem, due to Alon, Pach, and Solymosi, says that it is
preserved by vertex substitution~\cite{AlonPachSolymosi2001}.  Thus it
is enough to prove the conjecture for prime graphs, that is, graphs with
no nontrivial module.

The small prime graphs have been central to the subject.  The property
is classical for $P_4$; indeed, every $P_4$-free graph $G$ satisfies
$\alpha(G)\omega(G)\geq |G|$.  Thus all graphs on at most four vertices
are Erd\H{o}s--Hajnal.  Among prime five-vertex graphs, the bull was
settled by Chudnovsky and Safra~\cite{ChudnovskySafra2008}, and the
$5$-cycle by Chudnovsky, Scott, Seymour, and
Spirkl~\cite{ChudnovskyScottSeymourSpirkl2023}.  Nguyen, Scott, and
Seymour recently proved the conjecture for the five-vertex path
$P_5$~\cite{NguyenScottSeymourP5}.  Together with complementation and
the substitution theorem, this completes the conjecture for every graph
on at most five vertices.

\subsection{Six-vertex cases and the contribution of this paper}

Six is the first order for which the conjecture is not settled.  There
are $26$ prime graphs on six vertices, forming $13$ complementary
pairs.  A direct enumeration gives ten pairs
with a member containing an induced $P_5$ or $\overline{P_5}$, and
three with no such member.\footnote{The figures ten
and six in~\cite[p.~4]{NguyenScottSeymourDensityIV} count complementary
pairs and individual graphs, respectively.}

The symbols $N,J,F$ are local notation.  Let $N$ denote the left-hand
six-vertex graph in~\cite[Figure~1]{NguyenScottSeymourDensityIV}; after
relabelling its vertices as $a,b,c,d,e,f$, it has edge set
\[
             E(N)=\{ab,bc,bd,de,af,bf,df\}.
\]
The graph $J$ is the $P_5$-extension described
in~\cite[p.~22]{NguyenScottSeymourDensityIV}, while $E$ and the Bird are
the graphs in~\cite[Figures~4 and~5]{HuangJuZhou2026}.
For each of $J,E,\mathrm{Bird},F$, the vertices $a,b,c,d,e$ induce the
path $a-b-c-d-e$, and the neighbourhood of $f$ is
\[
 \begin{aligned}
 N_J(f)&=\{a,b,c,d\}, & N_E(f)&=\{c\},\\
 N_{\mathrm{Bird}}(f)&=\{b,c\}, & N_F(f)&=\{a,b,d\}.
 \end{aligned}
\]
Thus $J$ is obtained from $P_5$ by adding a vertex adjacent to every
path vertex except one end, $E$ is obtained from $P_5$ by adding a leaf
at its middle vertex, and the Bird is obtained by extending one horn of
the bull by one edge.  Equivalently, $F-e$ is the house and
$N_F(e)=\{d\}$, where $d$ has degree two in $F-e$ and lies outside its
unique triangle $abf$.
Figure~\ref{fig:six-vertex-representatives} shows these five graphs.

\begin{figure}[!t]
\centering
\begin{tikzpicture}[x=1.18cm,y=1.18cm,
  caseedge/.style={draw=black!86,line width=.58pt,
    line cap=round,line join=round},
  casevertex/.style={circle,draw=black!92,fill=white,line width=.55pt,
    minimum size=5.2mm,inner sep=0pt,outer sep=0pt,font=\footnotesize},
  casetitle/.style={font=\small}]
\begin{scope}
  \node[casetitle] at (0,.86) {$N$};
  \coordinate (Na) at (-1.050,-.120);
  \coordinate (Nb) at (-.525,.356);
  \coordinate (Nc) at (.525,.356);
  \coordinate (Nd) at (1.050,-.120);
  \coordinate (Ne) at (.525,-.596);
  \coordinate (Nf) at (-.525,-.596);
  \draw[caseedge] (Na)--(Nb) (Nb)--(Nc) (Nb)--(Nd) (Nd)--(Ne)
    (Na)--(Nf) (Nb)--(Nf) (Nd)--(Nf);
  \node[casevertex] at (Na) {$a$};
  \node[casevertex] at (Nb) {$b$};
  \node[casevertex] at (Nc) {$c$};
  \node[casevertex] at (Nd) {$d$};
  \node[casevertex] at (Ne) {$e$};
  \node[casevertex] at (Nf) {$f$};
\end{scope}
\begin{scope}[shift={(2.84,0)}]
  \node[casetitle] at (0,.86) {$J$};
  \coordinate (Ja) at (-1.050,-.120);
  \coordinate (Jb) at (-.525,.356);
  \coordinate (Jc) at (.525,.356);
  \coordinate (Jd) at (1.050,-.120);
  \coordinate (Je) at (.525,-.596);
  \coordinate (Jf) at (-.525,-.596);
  \draw[caseedge] (Ja)--(Jb)--(Jc)--(Jd)--(Je);
  \draw[caseedge] (Jf)--(Ja) (Jf)--(Jb) (Jf)--(Jc) (Jf)--(Jd);
  \node[casevertex] at (Ja) {$a$};
  \node[casevertex] at (Jb) {$b$};
  \node[casevertex] at (Jc) {$c$};
  \node[casevertex] at (Jd) {$d$};
  \node[casevertex] at (Je) {$e$};
  \node[casevertex] at (Jf) {$f$};
\end{scope}
\begin{scope}[shift={(5.68,0)}]
  \node[casetitle] at (0,.86) {$E$};
  \coordinate (Ea) at (-1.050,-.120);
  \coordinate (Eb) at (-.525,.356);
  \coordinate (Ec) at (.525,.356);
  \coordinate (Ed) at (1.050,-.120);
  \coordinate (Ee) at (.525,-.596);
  \coordinate (Ef) at (-.525,-.596);
  \draw[caseedge] (Ea)--(Eb)--(Ec)--(Ed)--(Ee) (Ec)--(Ef);
  \node[casevertex] at (Ea) {$a$};
  \node[casevertex] at (Eb) {$b$};
  \node[casevertex] at (Ec) {$c$};
  \node[casevertex] at (Ed) {$d$};
  \node[casevertex] at (Ee) {$e$};
  \node[casevertex] at (Ef) {$f$};
\end{scope}
\begin{scope}[shift={(8.52,0)}]
  \node[casetitle] at (0,.86) {$\mathrm{Bird}$};
  \coordinate (Ba) at (-1.050,-.120);
  \coordinate (Bb) at (-.525,.356);
  \coordinate (Bc) at (.525,.356);
  \coordinate (Bd) at (1.050,-.120);
  \coordinate (Be) at (.525,-.596);
  \coordinate (Bf) at (-.525,-.596);
  \draw[caseedge] (Ba)--(Bb)--(Bc)--(Bd)--(Be);
  \draw[caseedge] (Bf)--(Bb) (Bf)--(Bc);
  \node[casevertex] at (Ba) {$a$};
  \node[casevertex] at (Bb) {$b$};
  \node[casevertex] at (Bc) {$c$};
  \node[casevertex] at (Bd) {$d$};
  \node[casevertex] at (Be) {$e$};
  \node[casevertex] at (Bf) {$f$};
\end{scope}
\begin{scope}[shift={(11.36,0)}]
  \node[casetitle] at (0,.86) {$F$};
  \coordinate (Fa) at (-1.050,-.120);
  \coordinate (Fb) at (-.525,.356);
  \coordinate (Fc) at (.525,.356);
  \coordinate (Fd) at (1.050,-.120);
  \coordinate (Fe) at (.525,-.596);
  \coordinate (Ff) at (-.525,-.596);
  \draw[caseedge] (Fa)--(Fb)--(Fc)--(Fd)--(Fe);
  \draw[caseedge] (Ff)--(Fa) (Ff)--(Fb) (Ff)--(Fd);
  \node[casevertex] at (Fa) {$a$};
  \node[casevertex] at (Fb) {$b$};
  \node[casevertex] at (Fc) {$c$};
  \node[casevertex] at (Fd) {$d$};
  \node[casevertex] at (Fe) {$e$};
  \node[casevertex] at (Ff) {$f$};
\end{scope}
\end{tikzpicture}
\caption{The graphs $N,J,E,\mathrm{Bird},F$ are all
Erd\H{o}s--Hajnal.  The cases $N,J$ were known, Huang--Ju--Zhou proved
$E$ and the Bird in a preprint, and this paper proves the new
case $F$.}
\label{fig:six-vertex-representatives}
\end{figure}
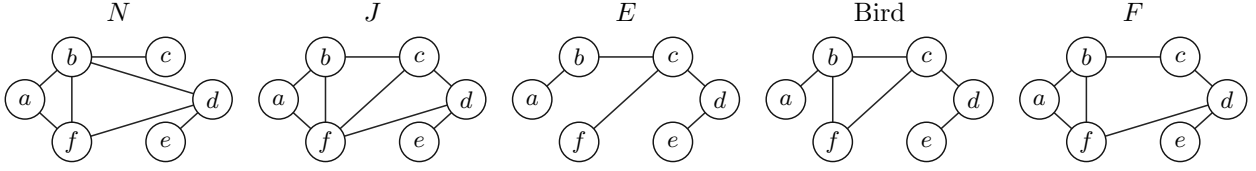%

\Needspace{12\baselineskip}
\subsection{Main theorem and proof strategy}

Our main result is the following.

\begin{theorem}
\label{thm:main}
The graph $F$ is Erd\H{o}s--Hajnal.
\end{theorem}

The proof uses the iterative-sparsification method of Nguyen, Scott, and
Seymour~\cite{NguyenScottSeymourDensityIV}, in the comb-based framework
developed by Huang, Ju, and Zhou~\cite{HuangJuZhou2026}.  In that
framework, leaf-reducibility together with a comb statement called
property $(*)$ yields generalized niceness; generalized niceness,
leaf-reducibility, and wonderfulness then yield the Erd\H{o}s--Hajnal
property.  We use this architecture and several of its ideas.

Some inputs and reduction steps are needed here in relative-size or
complement-symmetric forms.  A few passages in the preprint were not
clear to us in exactly those forms, and we may have missed an intended
convention or normalisation.  We therefore cite the inputs used
unchanged and state and prove the variants and reduction steps needed
here.  This concerns only the general framework: our proof of $F$
neither uses nor reproves the graph-specific arguments for the
$E$-graph or the Bird, and those results remain entirely due to Huang,
Ju, and Zhou.

Put $T=\overline F$.  Section~\ref{sec:target-local} verifies three
preliminary facts.  First, $\{F\}$ is leaf-reducible because $e$ is a
leaf of $F$ and $F-e=\overline{P_5}$.  Second, the leaf/co-leaf closure
and the known $P_5$ theorem imply that the family $\{F,T\}$ is
Erd\H{o}s--Hajnal; this family result is used later for one pure
blockade pattern.  Third, $\{F\}$ satisfies the bidirectional form of
wonderfulness required by the reduction.

We do not prove property $(*)$ for $F$.  Instead, we use the
quantitative comb condition in Definition~\ref{def:comb-hypothesis},
whose pure outcome may retain only a fixed positive power of the teeth,
rather than all of them.  Concretely,
Theorem~\ref{thm:F-comb-conclusion} gives constants $b,g,p,q>0$ such
that every equipped $(\ell,w)$-comb in a $T$-free graph, with
$\ell,w\geq4$, has one of the following two outcomes:
\[
\begin{array}{ll}
\mathrm{(U)} & \text{a complete or anticomplete $(k,w/k^b)$-blockade
                     with $k\geq\ell^g$};\\[2pt]
\mathrm{(P)} & \text{a pure $(k,w/\ell^q)$-blockade with
                     $\ell^p\leq k\leq\ell$}.
\end{array}
\]
Here ``equipped'' means that a vertex outside the comb is complete to
every tooth and anticomplete to every handle.  The homogeneous-set
outcome allowed in Definition~\ref{def:comb-hypothesis} is not needed
for $F$.  Section~\ref{sec:hjz-quantitative} proves, by adapting the
five-stage reduction of Huang--Ju--Zhou, that this weaker quantitative
condition still implies generalized niceness for a leaf-reducible
family.

We next describe the graph-specific part of the proof.  Write $h_i$ and
$B_i$ for the handles and teeth of the comb.  For sufficiently large
$\ell$, a pivot--true-twin extraction gives either many clique handles,
or a polynomial-size index set $I$ on which every handle edge has an
anticomplete tooth interface.  In the clique case, a propagation lemma
for induced copies of $T$ and the multipartite aggregation lemma of
Section~\ref{sec:multipartite} give outcome~(U).

In the other case, the component--anticomponent extraction either gives
outcome~(U) immediately or, for every $i\in I$, produces an
anticonnected set $X_i$ contained in a connected carrier
$C_i\subseteq B_i$.  Only after this extraction do the one-edge and
three-edge rectangle exclusions show that every retained interface
$X_i$--$X_j$ is additively separable over $\mathbb F_2$.  Interfaces
over handle edges are already constant zero, so every mixed interface
lies over a handle nonedge.

The rest of the argument is organised by two auxiliary graphs.  The
first, $M$, records the mixed interfaces.  Every hole of $M$ has even
length, and hence $M$ is $C_5$-free.  The Erd\H{o}s--Hajnal theorem for
$C_5$ supplies a polynomial-size set on which all interfaces are pure
or all are mixed.  The all-pure case gives outcome~(P).  In
the all-mixed case, the corresponding handles are pairwise nonadjacent,
and a second graph $K$ records the interfaces whose two endpoint
functions are both nonconstant.  We prove that $K$ is a cluster graph.

A large component of $K$ has synchronised endpoint cuts.  Choosing one
fibre in each corresponding tooth produces a pure blockade whose
pattern is $\{F,T\}$-free; the Erd\H{o}s--Hajnal property of the family
$\{F,T\}$ then gives outcome~(U).  If every component of $K$
is small, we choose one index from each of many components.  The
remaining interfaces are one-sided, their orientations are transitive,
and their endpoint cuts are nested.  The resulting pure pattern is an
interval graph, whose clique-or-stable-set bound again gives
outcome~(U).

Section~\ref{sec:main-comb} carries out the exponent bookkeeping and
handles bounded $\ell$ separately, proving the quantitative comb
condition for $\{F\}$.  Lemma~\ref{lem:hjz-poly-length} then combines
this condition with leaf-reducibility to make $\{F\}$ generalized nice.
Finally, Theorem~\ref{thm:hjz-generalized-nice} combines generalized
niceness with leaf-reducibility and wonderfulness to prove
Theorem~\ref{thm:main}.

To place Theorem~\ref{thm:main} in the six-vertex picture, the graphs
$N$ and $J$ are prime
\cite[Figure~1 and p.~22]{NguyenScottSeymourDensityIV}.
The $E$-graph and the Bird are prime by
Lemma~\ref{lem:prime-leaf-extension}, applied to $P_5$ and the bull; in
neither graph is the attachment vertex the neighbour of a leaf.  These
base graphs are connected and prime~\cite[p.~3]{HuangJuZhou2026}, and
$F$ is prime by Lemma~\ref{lem:F-prime}.  The five graphs and their
complements have pairwise distinct degree sequences.  Consequently, the
Huang--Ju--Zhou theorems for $E$ and the Bird, the cited results for
$N,J$, and Theorem~\ref{thm:main} settle five of the thirteen
complementary pairs.
Of the eight remaining pairs, two have no member containing an induced
$P_5$ or $\overline{P_5}$, while six do.  To the best of our knowledge,
all eight remain open as of 26 August 2026.

The paper is organised as follows.  Section~\ref{sec:preliminaries}
records the terminology and the general reductions.
Section~\ref{sec:target-local} verifies the framework inputs and develops
the local rectangle and affine-interface lemmas.
Sections~\ref{sec:all-mixed} and~\ref{sec:outer-graph} analyse the two
auxiliary graphs, and Section~\ref{sec:multipartite} proves the
clique-handle aggregation lemma.  Section~\ref{sec:hjz-quantitative}
gives the quantitative reductions in the forms used here.  Finally,
Section~\ref{sec:main-comb} assembles the structural alternatives,
verifies the quantitative comb condition, and proves
Theorem~\ref{thm:main}.

\section{Preliminaries and the iterative-sparsification framework}
\label{sec:preliminaries}

All graphs are finite, simple, and nonempty unless explicitly stated
otherwise.  We write $V(G)$ for the vertex set
of a graph $G$, $|G|=|V(G)|$, $G[X]$ for the subgraph induced by
$X\subseteq V(G)$, and $\overline G$ for the complement of $G$.  For a
positive integer $k$, put $[k]=\{1,\ldots,k\}$.  A graph is
$\mathcal H$-free, for a family $\mathcal H$, if it is $H$-free for
every $H\in\mathcal H$.  We put
$\overline{\mathcal H}=\{\overline H:H\in\mathcal H\}$ and
\[
             \operatorname{hom}(G)=
             \max\{\alpha(G),\omega(G)\}.
\]
A finite family $\mathcal H$ has the Erd\H{o}s--Hajnal property if
there is $c>0$ such that every $\mathcal H$-free graph $G$ satisfies
$\operatorname{hom}(G)\geq |G|^c$.

Let $X,Y$ be disjoint vertex sets.  We say that $X$ is
\emph{complete} to $Y$ if every pair in $X\times Y$ is an edge, and
\emph{anticomplete} to $Y$ if no such pair is an edge.  The pair
$(X,Y)$ is \emph{pure} if it is complete or anticomplete and is
\emph{mixed} otherwise.  A graph is \emph{anticonnected} if its
complement is connected; an \emph{anticomponent} is a connected
component of the complement.

For $x\in(0,1)$, the set $X$ is \emph{$x$-sparse to $Y$} if every
vertex of $X$ has at most $x|Y|$ neighbours in $Y$.  The pair $(X,Y)$ is
\emph{weakly $x$-sparse} if
$|E(X,Y)|\leq x|X||Y|$.  A graph $G$ is \emph{$x$-restricted} if
$G$ or $\overline G$ has maximum degree at most $x|G|$.

\subsection{Blockades and combs}

We use the terminology of Nguyen--Scott--Seymour and
Huang--Ju--Zhou~\cite{HuangJuZhou2026,NguyenScottSeymourDensityIV}.
An \emph{$(\ell,w)$-blockade} in $G$ is a sequence
$\mathcal B=(B_1,\ldots,B_k)$ of pairwise disjoint subsets of $V(G)$
such that $k\geq\ell$ and $|B_i|\geq w$ for every $i$.  Thus $\ell$
and $w$ are lower bounds and need not be integers, although the actual
number and orders of the blocks are integers.  The blockade is
\emph{complete}, \emph{anticomplete}, or \emph{pure} when every pair
of distinct blocks has the corresponding property.  It is
\emph{$x$-sparse} if $B_i$ is $x$-sparse to $B_j$ whenever $i>j$.
When $\mathcal B$ is pure, its \emph{pattern} is the graph on $[k]$ in
which $ij$ is an edge precisely when $B_i$ is complete to $B_j$.

For a positive integer $\ell$ and a real number $w>0$, an
\emph{$(\ell,w)$-comb} is a sequence
\[
                     ((h_i,B_i):i\in[\ell])
\]
such that $(B_1,\ldots,B_\ell)$ is an $(\ell,w)$-blockade, the
vertices $h_1,\ldots,h_\ell$ are distinct and disjoint from all the
blocks, and, for distinct $i,j$, the vertex $h_i$ is complete to
$B_i$ and anticomplete to $B_j$.  We call the $h_i$ the
\emph{handles} and the $B_i$ the \emph{teeth}.  Property $(*)$ below
also assumes a vertex
\begin{equation}
 \begin{split}
 r&\notin\{h_1,\ldots,h_\ell\}\cup\bigcup_iB_i,\\
 r&\text{ is complete to every tooth and anticomplete to every handle.}
 \end{split}
 \label{eq:root-outside}
\end{equation}
We will always say explicitly that the comb is equipped with such a
vertex $r$.

For two teeth $X_i\subseteq B_i$ and $X_j\subseteq B_j$, their
\emph{interface} is the bipartite adjacency relation between them.  Its
adjacency matrix is denoted by
\[
       M_{ij}(x,y)=1_{xy\in E(G)}\qquad
       (x\in X_i,\ y\in X_j).
\]
All additions of entries of such a matrix are in $\mathbb F_2$ and are
denoted by $\oplus$.

\subsection{Modules and substitution}

A set $S\subseteq V(H)$ is a \emph{module} if every vertex outside
$S$ is either complete or anticomplete to $S$.  It is nontrivial if
$1<|S|<|H|$, and $H$ is \emph{prime} if it has no nontrivial module.

\begin{lemma}[Prime leaf extension]
\label{lem:prime-leaf-extension}
Let $H$ be connected and prime, let $v\in V(H)$, and let $H^+$ be
obtained from $H$ by adding a leaf $x$ adjacent to $v$.  If $H$ has no
leaf whose neighbour is $v$, then $H^+$ is prime.
\end{lemma}

\begin{proof}
Suppose that $M$ is a proper nontrivial module of $H^+$.  If $x\notin
M$, then $M$ is a module of $H$, so primeness gives $M=V(H)$.  This is
impossible because $x$ is mixed on $V(H)$.

Now let $x\in M$.  If $v\in M$, every vertex of $H\setminus M$ is
anticomplete to $M$.  Connectedness of $H$ then gives $M=V(H^+)$, a
contradiction.  If $v\notin M$, put $S=M\setminus\{x\}$.  The vertex
$v$ is complete to $S$, and every other vertex of $H\setminus S$ is
anticomplete to $S$; hence $S$ is a module of $H$.  Since $v\notin S$
and $H$ is prime, $|S|=1$.  Its unique vertex is a leaf of $H$ adjacent
to $v$, contrary to the hypothesis.
\end{proof}

If $H_1,H_2$ are graphs and $v\in V(H_1)$, substitution of $H_2$ for
$v$ replaces $v$ by a copy of $H_2$ which is complete to every
neighbour of $v$ in $H_1$ and anticomplete to every nonneighbour of
$v$.  We use the following theorem.

\begin{theorem}[Alon--Pach--Solymosi \cite{AlonPachSolymosi2001}]
\label{thm:APS-substitution}
If $H_1$ and $H_2$ have the Erd\H{o}s--Hajnal property, then every graph
obtained by substituting $H_2$ for a vertex of $H_1$ also has the
Erd\H{o}s--Hajnal property.
\end{theorem}

We also use the following two established cases.

\begin{theorem}[Nguyen--Scott--Seymour
\cite{NguyenScottSeymourP5}]
\label{thm:P5-EH}
The graph $P_5$, and therefore also $\overline{P_5}$, has the
Erd\H{o}s--Hajnal property.
\end{theorem}

\begin{theorem}[Chudnovsky--Scott--Seymour--Spirkl
\cite{ChudnovskyScottSeymourSpirkl2023}]
\label{thm:C5-EH}
There is a constant $\tau>0$ such that every $C_5$-free graph $Q$
satisfies
\[
                  \operatorname{hom}(Q)\geq |Q|^\tau.
\]
\end{theorem}

\subsection{The Huang--Ju--Zhou framework}

We follow Huang--Ju--Zhou for the definitions that enter the
reduction~\cite{HuangJuZhou2026}, identifying explicitly the variants
used below.  A finite family $\mathcal F$ is
\emph{leaf-reducible} if some $H\in\mathcal F$ has a vertex $v$ of
degree one and
\[
                  (\mathcal F\setminus\{H\})\cup\{H-v\}
\]
has the Erd\H{o}s--Hajnal property.

A finite family $\mathcal F$ is \emph{generalized nice} if there are
constants
\[
 c_1\geq3,\quad c_2\geq8,\quad c_3,c_4,c_5,c_8>0,
 \quad c_6\geq1,\quad c_7\geq4
\]
such that, for every $\overline{\mathcal F}$-free graph $G$ and every
$0<\varepsilon<1/2$, at least one of the following holds:
\begin{enumerate}[label=(GN\arabic*)]
 \item $G$ has an
       $(\varepsilon^{-1},\varepsilon^{c_1}|G|)$-blockade whose blocks
       are pairwise complete or weakly $\varepsilon^{c_2}$-sparse;
 \item $G$ has a clique or stable set of order at least
       $(\varepsilon^{c_3}|G|)^{c_4}$;
 \item $G$ has a complete or anticomplete
       $(k,|G|/k^{c_5})$-blockade with
       $k\geq\varepsilon^{-c_6}$;
 \item $G$ has an $\varepsilon^{c_7}$-restricted induced subgraph of
       order at least $\varepsilon^{c_8}|G|$.
\end{enumerate}

The family $\mathcal F$ has \emph{property $(*)$} if there are
$d_1,d_2,d_3>0$ such that every $(\ell,w)$-comb in every
$\overline{\mathcal F}$-free graph, where $\ell,w\geq4$, equipped with
a vertex $r$ as in \eqref{eq:root-outside}, has one of the following
outcomes:
\begin{enumerate}[label=(\ensuremath{*}\arabic*)]
 \item $\operatorname{hom}(G)\geq w^{d_1}$;
 \item a complete or anticomplete $(k,w/k^{d_2})$-blockade with
       $k\geq\ell^{d_3}$;
 \item a pure $(\ell,w/\ell^2)$-blockade.
\end{enumerate}

For the applications below, we use the following bidirectional variant
of wonderfulness.  The family $\mathcal F$
is \emph{wonderful} if there is $a\geq6$ such that, for every
$y\in(0,1/2)$ and every $\overline{\mathcal F}$-free graph $G$, the
following holds.  If
$\mathcal B=(B_1,\ldots,B_\ell)$ is an $(\ell,w)$-blockade with
$\ell\geq y^{-a}$, all blocks are anticonnected and have equal order,
and, for all distinct $i,j$, either $B_i$ is complete to $B_j$ or
$B_i$ is $y^a$-sparse to $B_j$, then either
\begin{enumerate}[label=(W\arabic*)]
 \item $G$ has a $y^4$-restricted induced subgraph of order at least
       $w$; or
 \item for some $i\in[\ell]$, at most $y|G|$ vertices
       $v\in V(G)\setminus\bigcup_jB_j$ satisfy
       $0<|N(v)\cap B_i|<|B_i|/2$.
\end{enumerate}
Here the quantifier over $i,j$ is over ordered pairs; consequently
every noncomplete pair is sparse in both directions.

The framework of Huang--Ju--Zhou is organised around the following two
implications~\cite{HuangJuZhou2026}:
\begin{enumerate}
 \item If $\mathcal F$ is leaf-reducible and has property $(*)$, then
       $\mathcal F$ is generalized nice.
 \item If $\mathcal F$ is leaf-reducible, wonderful, and generalized
       nice, then $\mathcal F$ has the Erd\H{o}s--Hajnal property.
\end{enumerate}
To make the variants used below explicit,
Section~\ref{sec:hjz-quantitative} proves the relative leaf lemma, the
iterations in the precise forms used here, a quantitative extension of
the first implication, and a direct version of the second.
These general statements are applied here only to $F$; they are not
graph-specific proofs of the $E$-graph or Bird cases.

We shall use the following sufficient condition for wonderfulness.  Let
$H$ be a graph with distinct special vertices $v_1,v_2$.  Define $H^+$
by adding a new vertex adjacent only to $v_1,v_2$ and, if necessary,
adding the edge $v_1v_2$.  Define $H^-$ by adding a new vertex adjacent
only to $v_1,v_2$ and, if necessary, deleting the edge $v_1v_2$.

\begin{lemma}[Wonderfulness criteria]
\label{lem:HJZ-wonderful-criterion}
Suppose that one of the following holds.
\begin{enumerate}[label=(\roman*)]
\item Some member of $\mathcal F$ is an induced subgraph of the
      $1$-subdivision of $K_{1,t}$ for some $t\geq1$.
\item There is a graph $H$ with two special vertices for which
      $\{H\}\cup\overline{\mathcal F}$ has the Erd\H{o}s--Hajnal
      property and neither $H^+$ nor $H^-$ is
      $\overline{\mathcal F}$-free.
\end{enumerate}
Then $\mathcal F$ is wonderful.
\end{lemma}

\begin{proof}
We prove a bidirectional version of
\cite[Lemma~2.1]{HuangJuZhou2026} sufficient for our application.  This
may coincide with the intended convention there; here the ordered-pair
quantifier is explicit.  In case~(i), put
$d=1$ and
$\kappa=1/(2t)$.  In case~(ii), let $\kappa>0$ be an
Erd\H{o}s--Hajnal exponent for
$\{H\}\cup\overline{\mathcal F}$ and put
\[
 d=\max\{|Q|:Q\in\{H\}\cup\overline{\mathcal F}\}.
\]
Choose $a\geq\max\{6,1+5/\kappa\}$ so large that
$(d-1)2^{-a}<1/2$.  We verify wonderfulness with this $a$.

Fix $0<y<1/2$, a $\overline{\mathcal F}$-free graph $G$, and a
blockade $(B_1,\ldots,B_\ell)$ satisfying the hypotheses in the
definition.  Write $W=|B_i|$, and, for every vertex
$v\notin\bigcup_iB_i$, put
\[
 I(v)=\{i:0<|N(v)\cap B_i|<W/2\}.
\]
If $|I(v)|\leq y\ell$ for every such $v$, double counting the pairs
$(v,i)$ with $i\in I(v)$ gives an index $i$ for which at most
$y|G|$ outside vertices are counted.  This is outcome~(W2).

We may therefore choose $v$ with $|I(v)|>y\ell$.  Let $J$ be the
graph on $I(v)$ in which $ij$ is an edge exactly when $B_i$ is
complete to $B_j$.  We claim that $J$ has a clique or stable set of
order at least $|J|^\kappa$.

In case~(i), $J$ has no clique of order $t$.  Indeed,
anticonnectedness of $B_i$ gives a nonedge $b_ib'_i$ in $B_i$ with
$vb_i$ an edge and $vb'_i$ a nonedge: take an edge of
$\overline{G[B_i]}$ crossing the nonempty sets $N(v)\cap B_i$ and
$B_i\setminus N(v)$.  A clique on $t$ pattern vertices, together with
these pairs and $v$, would induce in $G$ the complement of the
$1$-subdivision of $K_{1,t}$.  It would therefore contain an induced
member of $\overline{\mathcal F}$, a contradiction.  If $t=1$, this
contradicts the nonemptiness of $J$, so the assumption that (W2) fails
is impossible and (W2) holds.  Hence assume $t\geq2$, and put
$n=|J|$.  If $n<2^{2t}$, a clique or stable set of order two (or the sole
vertex when $n=1$) has order at least $n^{1/(2t)}$.  If
$n\geq2^{2t}$, set $s=\lceil n^{1/(2t)}\rceil$.  The standard bound
$R(t,s)\leq s^t$, together with
\[
                 s^t\leq2^t n^{1/2}\leq n,
\]
gives a stable set of order $s$ because $J$ is $K_t$-free.  In either
case, $J$ has a clique or stable set of order at least
$n^{1/(2t)}=n^\kappa$.

In case~(ii), we instead claim that $J$ is
$\{H\}\cup\overline{\mathcal F}$-free.  If an induced subgraph of $J$
is isomorphic to some $Q\in\overline{\mathcal F}$, choose
representatives of its pattern vertices successively.  An edge of the
pattern joins complete blocks.  For each required nonedge to an earlier
representative, the ordered-pair sparsity hypothesis excludes at most
$y^aW$ vertices of the new block.  Since fewer than $d$
representatives have already been chosen and $(d-1)y^a<1$, the
selection can be completed.  The representatives induce $Q$ in $G$,
a contradiction.

If $J$ contains an induced copy of $H$, index its two special pattern
vertices first.  Choose their representatives in $N(v)$; this is
possible because their indices belong to $I(v)$.  For every other
pattern vertex, choose a representative outside $N(v)$ and realise all
required nonedges to earlier representatives.  At each step, fewer
than
\[
                 \bigl(1/2+(d-1)y^a\bigr)W<W
\]
vertices are excluded.  Pattern edges are automatic because the
corresponding blocks are complete.  The selected representatives,
together with $v$, therefore induce $H^+$ or $H^-$, according to the
adjacency of the first two representatives.  Each of these two graphs
contains an induced member of $\overline{\mathcal F}$, contrary to the
choice of $G$.  Thus $J$ is
$\{H\}\cup\overline{\mathcal F}$-free, and its Erd\H{o}s--Hajnal
property proves the claim.

Let $R$ be the clique or stable set just obtained, and write
\[
                    r=|R|\geq |I(v)|^\kappa>(y\ell)^\kappa.
\]
Put $S=\bigcup_{i\in R}B_i$.  If $R$ is stable, every noncomplete
pair of its blocks is $y^a$-sparse in both directions, and hence
\[
             \Delta(G[S])\leq(r^{-1}+y^a)|S|.
\]
If $R$ is a clique, all its blocks are pairwise complete, and
\[
             \Delta(\overline{G[S]})\leq r^{-1}|S|.
\]
Since $\ell\geq y^{-a}$,
\[
 r^{-1}< (y\ell)^{-\kappa}
       \leq y^{\kappa(a-1)}\leq y^5,
 \qquad y^a\leq y^6.
\]
Thus in either case $G[S]$ is $y^4$-restricted, because
$2y^5\leq y^4$.  Finally $|S|=rW\geq W\geq w$, so outcome~(W1) holds.
\end{proof}

We also need the leaf/co-leaf family closure in Erd\H{o}s--Hajnal form.  It is
Corollary~1.8 of~\cite{HuangJuZhou2026}, obtained from the corresponding
viral theorem of Nguyen--Scott--Seymour
\cite[Theorem~7.8]{NguyenScottSeymourDensityIV} and the standard
finite-family extension of the Erd\H{o}s--Hajnal/viral equivalence of
Buci\'c, Fox, and Pham~\cite[Theorem~4]{BucicFoxPham2024}; this extension
is stated in~\cite[Theorem~1.3]{HuangJuZhou2026}.

\begin{theorem}[Leaf/co-leaf family closure]
\label{thm:leaf-coleaf-closure}
Let $\mathcal Q$ be a finite graph family.  Suppose
$F_1,\overline{F_2}\in\mathcal Q$, and let $v_i$ be a leaf of $F_i$.
If both families
\[
 \{F_1-v_1\}\cup(\mathcal Q\setminus\{F_1\})
 \quad\text{and}\quad
 \{\overline{F_2-v_2}\}\cup
       (\mathcal Q\setminus\{\overline{F_2}\})
\]
have the Erd\H{o}s--Hajnal property, then so does $\mathcal Q$.
\end{theorem}

\section{The target graph and its local structure}
\label{sec:target-local}

Recall that $F$ has vertex set $\{a,b,c,d,e,f\}$ and edge set
\[
 E(F)=\{ab,bc,cd,de,af,bf,df\}.
\]
Put $T=\overline F$.  Thus
\begin{equation}
 E(T)=\{ac,ad,ae,bd,be,ce,cf,ef\}.
 \label{eq:T-edge-set}
\end{equation}
The seven nonedges of $T$ are
\begin{equation}
 ab,af,bc,bf,cd,de,df.
 \label{eq:T-nonedge-set}
\end{equation}
In particular, the degrees of $a,b,c,d,e,f$ in $T$ are respectively
\[
                         3,2,3,2,4,2.
\]
The vertex $e$ is the unique vertex of degree four, its unique nonneighbour
is $d$, and $d$ has degree two.

\begin{lemma}
\label{lem:F-prime}
The graphs $F$ and $T$ are prime.
\end{lemma}

\begin{proof}
It is enough to prove the assertion for $F$, since complementation
preserves modules.  Suppose that $S$ is a proper module of $F$ with at
least two vertices.  If $e\in S$ and $d\notin S$, then $d$ is adjacent
to $e$ and hence complete to $S$.  Thus $S\subseteq\{c,e,f\}$.  But $b$
distinguishes $c$ from $e$, while $a$ distinguishes $f$ from $e$, a
contradiction.  If $e,d\in S$, then $c,b,a,f$ are forced into $S$ in
that order: each displayed vertex distinguishes two vertices already
known to belong to $S$.  This gives $S=V(F)$, again a contradiction.

Suppose now that $e\notin S$.  If $d\in S$, the vertex $e$, which is
adjacent only to $d$, distinguishes $d$ from every possible second
member of $S$.
Thus $d\notin S$, and $S\subseteq\{a,b,c,f\}$.  The outside vertex $d$
is adjacent to $c,f$ and nonadjacent to $a,b$, so $S$ is contained in one of
these pairs.  The pair $\{c,f\}$ is distinguished by $a$, and
$\{a,b\}$ is distinguished by $c$.  No proper nontrivial module
exists.
\end{proof}

We use the definitions of blockades and combs from
Section~\ref{sec:preliminaries}.  Throughout the structural argument,
$((h_i,B_i):i\in[\ell])$ is an $(\ell,w)$-comb in a $T$-free graph,
and $r$ is a vertex outside the comb which is complete to every tooth
and anticomplete to every handle.  No condition is imposed on the
handle graph or on edges inside or between the teeth.

\subsection{The hypotheses from the iterative-sparsification framework}

\begin{proposition}[Leaf reduction and the complementary pair]
\label{prop:f124-leaf-pair}
The singleton $\{F\}$ is leaf-reducible.  The family $\{F,T\}$ has the
Erd\H{o}s--Hajnal property.
\end{proposition}

\begin{proof}
The vertex $e$ is a leaf of $F$.  On $\{a,b,c,d,f\}$, the complement of
$F-e$ is the path
\[
                         f-c-a-d-b.
\]
Thus $F-e=\overline{P_5}$, the house, and is Erd\H{o}s--Hajnal.  This is
exactly leaf-reducibility for the singleton family.

For the second assertion, apply
Theorem~\ref{thm:leaf-coleaf-closure} to
$\mathcal Q=\{F,T\}$ with $F_1=F_2=F$ and $v_1=v_2=e$.  The two
reduced families are exactly
\[
                  \{\overline{P_5},T\}
       \quad\hbox{and}\quad
                  \{P_5,F\}.
\]
A graph free of the first family is in particular
$\overline{P_5}$-free, and a graph free of the second is in particular
$P_5$-free.  Theorem~\ref{thm:P5-EH}, together with complementation,
therefore gives both hypotheses of the closure theorem.  Thus
$\{F,T\}$ has the Erd\H{o}s--Hajnal property using only the established
result for $P_5$.
\end{proof}

\begin{proposition}[Wonderfulness]
\label{prop:f124-wonderful}
The singleton family $\{F\}$ is wonderful in the sense of
the directionally precise definition in Section~\ref{sec:preliminaries}.
\end{proposition}

\begin{proof}
We verify Lemma~\ref{lem:HJZ-wonderful-criterion}.  Let $H$ be the
graph on $\{0,1,2,3,4,5\}$ with
\begin{equation}
 E(H)=\{02,03,04,05,12,13,14,23\},
 \label{eq:wonderful-H-edge-set}
\end{equation}
and take $v_1=2,v_2=5$ as its two special vertices.  The pair $\{2,3\}$
is a module: both vertices see $0,1$, miss $4,5$, and are adjacent to one
another.  Contracting this module gives a five-vertex graph, and $H$ is
obtained from that quotient by substituting $K_2$ for the contracted
vertex.  As recalled in Section~\ref{sec:introduction}, every graph on
at most five vertices has the Erd\H{o}s--Hajnal property
\cite{ChudnovskySafra2008,ChudnovskyScottSeymourSpirkl2023,
NguyenScottSeymourP5}; hence Theorem~\ref{thm:APS-substitution}
shows that $H$ has the
Erd\H{o}s--Hajnal property.  Consequently the family $\{H,T\}$ has that
property as well, since every $\{H,T\}$-free graph is $H$-free.

Let $z=6$ be the new vertex in the definitions of $H^+$ and $H^-$.  Since
$25\notin E(H)$, the graph $H^+$ adds the three edges $26,56,25$, whereas
$H^-$ adds $26,56$ and leaves $25$ absent.

In $H^+$, the map
\begin{equation*}
 (a,b,c,d,e,f)\longmapsto(0,1,5,4,2,6)
\end{equation*}
is an induced copy of $T$.  Indeed, the eight edges in
\eqref{eq:T-edge-set} become
\[
 05,04,02,14,12,25,56,26,
\]
all present, while the seven pairs in \eqref{eq:T-nonedge-set} become
\[
 01,06,15,16,45,24,46,
\]
all absent.

In $H^-$, the map
\begin{equation*}
 (a,b,c,d,e,f)\longmapsto(0,6,3,5,2,1)
\end{equation*}
is an induced copy of $T$.  The eight edges in
\eqref{eq:T-edge-set} become
\[
 03,05,02,65,62,32,31,21,
\]
and the seven nonedges in \eqref{eq:T-nonedge-set} become
\[
 06,01,63,61,35,52,51.
\]
Their stated signs follow directly from \eqref{eq:wonderful-H-edge-set}
and the definition of $H^-$.
Thus neither $H^+$ nor $H^-$ is $T$-free, and
Lemma~\ref{lem:HJZ-wonderful-criterion} applies.
\end{proof}

It remains to verify the comb conclusion required by the
iterative-sparsification framework of~\cite{HuangJuZhou2026}.

\subsection{Rectangle lemmas and affine interfaces}

Throughout the following statements, the host graph is $T$-free and
the comb is equipped with the vertex $r$ fixed above.  Each
contradiction is certified by an explicit induced copy of $T$, with all
eight edges and seven nonedges listed.

\begin{lemma}[One-edge rectangle lemma]
\label{lem:f124-one-edge-rectangle}
Let $x,x'\in B_i$ be nonadjacent.  Let $y,z\notin B_i$ be adjacent tooth
vertices (they may belong to the same tooth or to two different teeth).
Then the four pairs between $\{x,x'\}$ and $\{y,z\}$ cannot contain
exactly one edge.
\end{lemma}

\begin{proof}
After interchanging $x,x'$ and $y,z$, suppose that $xy$ is the unique
edge.  Then
\begin{equation*}
 (a,b,c,d,e,f)\longmapsto(x,x',y,h_i,r,z)
\end{equation*}
is an induced copy of $T$.  The eight required edges are
\[
 xy,xh_i,xr,x'h_i,x'r,yr,yz,rz,
\]
and the seven required nonedges are
\[
 xx',xz,x'y,x'z,yh_i,h_ir,zh_i.
\]
Here $h_ir$ is absent by \eqref{eq:root-outside}, while $yh_i$ and
$zh_i$ are absent because $y,z$ lie in teeth other than $B_i$.
\end{proof}

\begin{lemma}[Three-edge rectangle lemma]
\label{lem:f124-three-edge-rectangle}
Suppose that $h_ih_j$ is a nonedge.  If $x,x'\in B_i$ are nonadjacent
and $y,y'\in B_j$ are adjacent, then the four pairs between
$\{x,x'\}$ and $\{y,y'\}$ cannot contain exactly three edges.
Consequently that rectangle has an even number of edges.
\end{lemma}

\begin{proof}
Lemma~\ref{lem:f124-one-edge-rectangle} excludes one edge.  For three
edges, relabel so that $x'y'$ is the unique nonedge.  The map
\begin{equation*}
 (a,b,c,d,e,f)\longmapsto(x,x',y',h_i,y,h_j)
\end{equation*}
sends \eqref{eq:T-edge-set} to
\[
 xy',xh_i,xy,x'h_i,x'y,y'y,y'h_j,yh_j,
\]
and sends \eqref{eq:T-nonedge-set} to
\[
 xx',xh_j,x'y',x'h_j,y'h_i,h_iy,h_ih_j.
\]
Thus it gives an induced $T$, a contradiction.
\end{proof}

Here is the exact linear-algebra consequence.  We identify adjacency
values with elements of $\mathbb F_2$.  For fixed ordered rows $x,x'$,
the column of a vertex $y$ is
$(1_{xy\in E(G)},1_{x'y\in E(G)})^{\mathsf T}$.  We write $00$ for the
zero column and call $01$ and $10$ singleton columns.  A nonconstant
binary function partitions its domain into its two nonempty fibres.

\begin{lemma}[Affine representation of an interface]
\label{lem:f124-affine-cut}
Suppose $h_ih_j$ is a nonedge, $A\subseteq B_i$ is anticonnected, and
$C\subseteq B_j$ is connected.  There are functions
\[
 \lambda_A:A\longrightarrow\mathbb F_2,
 \qquad
 \lambda_C:C\longrightarrow\mathbb F_2
\]
and a sign $\gamma\in\mathbb F_2$ such that
\begin{equation}
  1_{ac\in E(G)}=\lambda_A(a)\oplus\lambda_C(c)\oplus\gamma
       \qquad(a\in A,c\in C).
 \label{eq:f124-affine-representation}
\end{equation}
In particular, some subset of at least $|A|/2$ is pure to some subset of
at least $|C|/2$.
\end{lemma}

\begin{proof}
By Lemmas~\ref{lem:f124-one-edge-rectangle} and
\ref{lem:f124-three-edge-rectangle}, every rectangle whose two rows are
a nonedge of $A$ and whose two columns are an edge of $C$ has even parity.
Join arbitrary $a,a'\in A$ by a path in $\overline{G[A]}$, and arbitrary
$c,c'\in C$ by a path in $G[C]$.  Summing the elementary rectangle
equalities over the Cartesian product of the edge sets of the two paths
in $\mathbb F_2$ cancels every internal term and gives
\[
 M(a,c)\oplus M(a,c')\oplus M(a',c)\oplus M(a',c')=0,
\]
where $M$ is the $A$--$C$ adjacency matrix.  Fix $a_0\in A,c_0\in C$
and take
\[
 \lambda_A(a)=M(a,c_0),\quad
 \lambda_C(c)=M(a_0,c),\quad
 \gamma=M(a_0,c_0).
\]
The displayed rectangle identity gives
\eqref{eq:f124-affine-representation}.  Taking a largest fibre of
each of the two binary functions gives the final pure pair.
\end{proof}

We next prove the synchronisation property for genuinely two-sided
mixed interfaces.  This is the step that avoids an exponential loss.

\begin{lemma}[Synchronisation of two-sided mixed interfaces]
\label{lem:f124-triple-sync}
Let $i,j,k$ have pairwise nonadjacent handles.  For $q\in\{i,j,k\}$,
let $X_q\subseteq B_q$ be anticonnected.  Suppose every
pairwise interface has a representation
\begin{equation}
 M_{pq}(x,y)=\lambda_p^q(x)\oplus\lambda_q^p(y)
                         \oplus\gamma_{pq},
 \label{eq:f124-triple-affine}
\end{equation}
and every one of the six displayed endpoint functions is nonconstant.
Then, at each tooth, its two partner functions agree up to complementation;
for example
\[
                         \lambda_i^j=\lambda_i^k\oplus\epsilon_i
\]
for a constant $\epsilon_i\in\mathbb F_2$.
\end{lemma}

\begin{proof}
We first prove that the two functions have the same difference on every
nonedge $xx'$ of $X_i$.  Suppose, to the contrary, that
\begin{equation}
 \lambda_i^j(x)\ne\lambda_i^j(x'),
 \qquad
 \lambda_i^k(x)=\lambda_i^k(x').
 \label{eq:f124-cut-discrepancy}
\end{equation}
For every $y\in X_j$, \eqref{eq:f124-triple-affine} and the first
relation in \eqref{eq:f124-cut-discrepancy} say that exactly one of
$xy,x'y$ is an edge.  Since $\lambda_k^i$ is nonconstant,
\eqref{eq:f124-triple-affine} and the second relation in
\eqref{eq:f124-cut-discrepancy} give a vertex $z\in X_k$ which is nonadjacent to both
$x,x'$ (one of the two fibres gives column $00$).  Finally, both endpoint
functions on the $X_j$--$X_k$ interface are nonconstant.  Consequently
every vertex of $X_k$ has a neighbour in $X_j$: its neighbourhood is one of
the two nonempty fibres of $\lambda_j^k$.  Choose such a neighbour $y$ of
$z$.

Now $yz$ is an edge, while the rectangle from $\{x,x'\}$ to $\{y,z\}$
has exactly one edge.  This contradicts
Lemma~\ref{lem:f124-one-edge-rectangle}.
Interchanging $j,k$ proves the reverse implication.  Thus the two
functions have equal differences across every edge of
$\overline{G[X_i]}$.  That graph is connected, so their pointwise sum is
constant.  The same argument at $X_j$ and $X_k$ completes the proof.
\end{proof}

\begin{corollary}[Synchronisation on a clique of two-sided interfaces]
\label{cor:f124-full-cut-extraction}
Let $I$ be a set of pairwise nonadjacent handles.  Suppose that every
$i\in I$ has an anticonnected set $X_i\subseteq B_i$, and
every interface $X_i$--$X_j$ has the form
\eqref{eq:f124-triple-affine} with both endpoint functions
nonconstant.  Assume $|I|\geq2$.  Then there are sets $X'_i\subseteq X_i$ with
$|X'_i|\geq |X_i|/2$ such that $(X'_i:i\in I)$ is a pure blockade.
\end{corollary}

\begin{proof}
Fix $i_0\in I$.  By Lemma~\ref{lem:f124-triple-sync}, for each $i\ne i_0$
all functions $\lambda_i^j$, $j\ne i$, agree up to complementation with
$\lambda_i^{i_0}$.  The same lemma applied to $(i_0,j,k)$ for any two
partners $j,k$ also synchronises all the functions $\lambda_{i_0}^j$.
(When $|I|=2$ this is immediate by choosing the one available function.)
Let $X'_i$ be a largest fibre of this common
bipartition.  Formula \eqref{eq:f124-triple-affine} is constant on
$X'_i\times X'_j$ for every
pair $i,j$, proving purity and the size bound.
\end{proof}

There is a second conclusion once the synchronised label classes have been
chosen.  It uses the Erd\H{o}s--Hajnal property of the pair $\{F,T\}$.

\begin{lemma}[A configuration excluding $F$ from the pattern]
\label{lem:f124-antitwin-quotient}
Let $A,B,C,D,E,F_0$ be vertices which induce $F$, in the displayed
roles; thus their edges are
\[
 AB,BC,CD,DE,AF_0,BF_0,DF_0.
\]
If a further vertex $x$ is nonadjacent to $B,F_0$ and adjacent to
$C,D,E$, then the graph
contains an induced $T$.  No assumption on $xA$ is needed.
\end{lemma}

\begin{proof}
The six vertices $B,C,D,E,F_0,x$ induce $T$ under
\begin{equation*}
 (a,b,c,d,e,f)\longmapsto(C,F_0,x,B,D,E).
\end{equation*}
Indeed, the eight pairs in \eqref{eq:T-edge-set} become
\[
 Cx,CB,CD,F_0B,F_0D,xD,xE,DE,
\]
all edges, and the seven pairs in \eqref{eq:T-nonedge-set} become
\[
 CF_0,CE,F_0x,F_0E,xB,BD,BE,
\]
all nonedges.
\end{proof}

\begin{corollary}[The synchronised pattern is $\{F,T\}$-free]
\label{cor:f124-sync-quotient-pairfree}
Under the hypotheses of Corollary~\ref{cor:f124-full-cut-extraction}, choose
each $X'_i$ to be one fibre of the synchronised nonconstant label on
$X_i$, and let $P$ be the pure-blockade pattern.  Then $P$ is
$\{F,T\}$-free.  Consequently $P$ has a clique or stable set of order at
least $|I|^\eta$ for a fixed $\eta>0$.
\end{corollary}

\begin{proof}
The pattern is $T$-free, since representatives of its blocks would give
the same induced graph in the $T$-free host.  Suppose it contains an
induced $F$, and choose representatives
$A,B,C,D,E,F_0$ from the six corresponding blocks.  In the tooth
containing $A$, the other label fibre is nonempty; choose $x$ there.
Changing only that synchronised endpoint label flips its adjacency to
each of the other five representatives.  Thus $x$ is nonadjacent to
$B,F_0$ and adjacent to $C,D,E$, contrary to
Lemma~\ref{lem:f124-antitwin-quotient}.  Hence $P$ is also $F$-free.
Proposition~\ref{prop:f124-leaf-pair} supplies the exponent $\eta$.
\end{proof}

The remaining case, in which some mixed interfaces are one-sided, is
treated in Section~\ref{sec:all-mixed}.

\section{Pairwise mixed additively separable interfaces}
\label{sec:all-mixed}

Let $T=\overline F$.  Work in a $T$-free graph containing a comb
equipped with a vertex $r$ as in Section~\ref{sec:preliminaries}.
Let $I$ be a set of pairwise nonadjacent handles, and let every
$X_i\subseteq B_i$ be anticonnected.  Suppose that each
pairwise interface is mixed and has the affine form
\begin{equation}
 M_{ij}(x,y)=\lambda_i^j(x)\oplus\lambda_j^i(y)
                         \oplus\gamma_{ij}.
 \label{eq:all-mixed-affine}
\end{equation}
Since the interface is mixed, at least one of its two endpoint functions
is nonconstant.  Within this section, call $ij$ \emph{two-sided} if both endpoint
functions are nonconstant, and \emph{one-sided} otherwise.  In the one-sided case orient the pair
from its nonconstant endpoint to its constant endpoint.
This classification does not depend on the chosen affine
representation.  Indeed, if
$f(x)\oplus g(y)\oplus c=f'(x)\oplus g'(y)\oplus c'$ for every
$(x,y)$, then fixing one endpoint shows that $f\oplus f'$ and
$g\oplus g'$ are constant.  Thus $f$ is nonconstant exactly when $f'$
is, and similarly for $g$ and $g'$.

We use Lemma~\ref{lem:f124-one-edge-rectangle}: if
$x,x'\in B_i$ are nonadjacent and adjacent tooth vertices $y,z$ lie
outside $B_i$, then the rectangle
$\{x,x'\}\times\{y,z\}$ cannot have exactly one edge.  Its labelled
$T$ witness is
\begin{equation}
 (a,b,c,d,e,f)\longmapsto(x,x',y,h_i,r,z).
 \label{eq:all-mixed-one-edge-witness}
\end{equation}

\begin{lemma}[No two-sided--two-sided--one-sided triangle]
\label{lem:f124-type-no-p3}
There are no three indices $i,j,k$ for which $ij,jk$ are two-sided and $ik$
is one-sided.
\end{lemma}

\begin{proof}
Suppose first that the one-sided pair is oriented $k\to i$.  Since
$\lambda_i^j$ is nonconstant and $\overline{G[X_i]}$ is connected,
there is a nonedge $xx'$ of $G[X_i]$ whose ends have different
$\lambda_i^j$-values.  Formula~\eqref{eq:all-mixed-affine} says that,
for every $y\in X_j$,
exactly one of $xy,x'y$ is an edge.

On the pair $ik$, the function $\lambda_i^k$ is constant and
$\lambda_k^i$ is nonconstant.  Consequently one of the two nonempty
fibres of $\lambda_k^i$ consists of vertices whose columns on
$\{x,x'\}$ are $00$; choose $z$ in that fibre.  Since $jk$ is two-sided,
$\lambda_j^k$ is nonconstant.  For this fixed $z$,
formula~\eqref{eq:all-mixed-affine} makes
$N_{X_j}(z)$ one of the two nonempty fibres of $\lambda_j^k$.  Choose
$y\in N_{X_j}(z)$.

Now $yz$ is an edge, its column on $\{x,x'\}$ is a singleton, and the
$z$-column is $00$.  This is exactly the forbidden one-edge rectangle
\eqref{eq:all-mixed-one-edge-witness}.  If $ik$ is oriented $i\to k$,
interchange $i,k$ and repeat the
same argument.
\end{proof}

\begin{corollary}[The auxiliary graph is a cluster graph]
\label{cor:f124-type-cluster}
Let $K$ be the graph on $I$ in which $ij$ is an edge precisely when the
interface $ij$ is two-sided.  Then $K$ is $P_3$-free, and hence every
connected component of $K$ is a clique.  In particular $K$ is
$\{F,T\}$-free.
\end{corollary}

\begin{proof}
An induced $P_3$ in $K$ is precisely the configuration excluded by
Lemma~\ref{lem:f124-type-no-p3}.  A connected noncomplete graph contains
an induced $P_3$: take the first three vertices of a shortest path
between two nonadjacent vertices.  Thus every component of $K$ is a
clique.  Both $F$ and $T$ are connected and noncomplete, so neither
is an induced subgraph of a disjoint union of cliques.
\end{proof}

We also record the complete-or-anticomplete conclusion in the one-sided
case.  A pure blockade of only polynomial length is not, by itself, one
of the three outcomes of property $(*)$.

\begin{lemma}[The one-sided pattern is an interval graph]
\label{lem:f124-one-sided-interval}
Suppose $J=\{1,\ldots,s\}\subseteq I$ and every interface on $J$ is
one-sided.  Then there are $X'_i\subseteq X_i$ with
\begin{equation}
                         |X'_i|\geq |X_i|/s
 \label{eq:one-sided-cell-size}
\end{equation}
such that $(X'_i:i\in J)$ is a pure blockade and its pattern is an
interval graph.  Consequently the pattern has a clique or stable set
of order at least $\sqrt{s}$.
\end{lemma}

\begin{proof}
The one-edge rectangle lemma first implies that the orientation of
the one-sided pairs has no directed triangle.  Indeed, for a directed
triangle $i\to j\to k\to i$, choose a nonedge $xx'$ of $X_i$ crossing
the nonconstant $i\to j$ cut.  The $k\to i$ interface supplies a vertex
$z\in X_k$ with column $00$ on $\{x,x'\}$, and the $j\to k$ interface
supplies a neighbour $y\in X_j$ of $z$.  The $y$-column is a singleton,
contradicting \eqref{eq:all-mixed-one-edge-witness}.  Hence the
tournament is transitive; order $J$ so
that $i\to j$ exactly when $i<j$.

For $i<j$, put
\[
 S_i^j=\{x\in X_i:x\text{ is complete to }X_j\}.
\]
These sets are nested:
\begin{equation}
             S_i^{i+1}\supseteq S_i^{i+2}\supseteq\cdots
                         \supseteq S_i^s.
 \label{eq:one-sided-nesting}
\end{equation}
To verify \eqref{eq:one-sided-nesting}, suppose $i<j<k$ and
$x\in S_i^k\setminus S_i^j$.  The nonconstant cut on the source
$X_j$ of $j\to k$, together with anticonnectivity of $X_j$, gives a
nonedge $yy'$ crossing its two fibres.  Say $y$ is complete and $y'$ is
anticomplete to $X_k$.  For any $z\in X_k$, the outside pair $x,z$ is
adjacent, while the rectangle
$\{y,y'\}\times\{z,x\}$ has exactly the one edge $yz$, contradicting
\eqref{eq:all-mixed-one-edge-witness}.

The chain \eqref{eq:one-sided-nesting} has at most $s$ membership cells.
Choose a largest cell $X'_i$, which proves
\eqref{eq:one-sided-cell-size}.  Membership in every $S_i^j$ is constant on
$X'_i$, so the retained blocks are pairwise pure.  There is an integer
$t_i\in\{i,i+1,\ldots,s\}$ such that, for $i<j$,
\begin{equation}
             X'_i\text{ is complete to }X'_j
                    \quad\Longleftrightarrow\quad j\leq t_i.
 \label{eq:one-sided-threshold}
\end{equation}
(Take $t_i=i$ when the cell belongs to none of the sets in
\eqref{eq:one-sided-nesting}.)
Represent the $i$th block by the closed interval $[i,t_i]$.  For
$i<j$, these two intervals meet exactly when $j\leq t_i$, which by
\eqref{eq:one-sided-threshold}
is exactly when the two blocks are complete.  Thus their pattern is an
interval graph.

An interval graph $P$ satisfies
$\chi(P)=\omega(P)$: process intervals by increasing left endpoint and
give each the first colour not used by an interval containing that
endpoint; all such previously active intervals form a clique.  A largest
colour class is stable, so
$|P|\leq\alpha(P)\chi(P)=\alpha(P)\omega(P)$.  Therefore
$\max(\alpha(P),\omega(P))\geq\sqrt{|P|}=\sqrt{s}$.
\end{proof}

\begin{theorem}[Extraction from pairwise mixed interfaces]
\label{thm:f124-mixed-affine-extraction}
Let $m=|I|\geq4$ and suppose $|X_i|\geq W$ for every $i\in I$.  Let
$\eta>0$ be an Erd\H{o}s--Hajnal exponent for $\{F,T\}$.  Put
\begin{equation}
                         \delta=\min\{\eta/2,1/4\}.
 \label{eq:mixed-extraction-delta}
\end{equation}
Then $G$ has a complete or anticomplete blockade of length at least
$m^\delta$ and width at least $W/m$.
\end{theorem}

\begin{proof}
Let $K$ be the cluster graph from
Corollary~\ref{cor:f124-type-cluster}.  If a component $C$ of $K$ has
order at least $\sqrt m$, all its interfaces are two-sided.
Corollary~\ref{cor:f124-full-cut-extraction} gives one label fibre of
order at least $W/2$ in every tooth of $C$, and
Corollary~\ref{cor:f124-sync-quotient-pairfree} says that the resulting
pure-blockade pattern is $\{F,T\}$-free.
Proposition~\ref{prop:f124-leaf-pair} therefore gives a clique or stable
set of blocks of order at least
$|C|^\eta\geq m^{\eta/2}$.  Their width is at least $W/2\geq W/m$.

Otherwise every component of $K$ has order less than $\sqrt m$, so
$K$ has more than $\sqrt m$ components.  Choose one index from each of
$s\geq\sqrt m$ components.  All selected interfaces are one-sided.
Lemma~\ref{lem:f124-one-sided-interval} gives blocks of width at least
$W/s\geq W/m$ and a complete or anticomplete subblockade of length at
least $\sqrt s\geq m^{1/4}$.
Equation~\eqref{eq:mixed-extraction-delta} covers both cases.
\end{proof}

\begin{corollary}[Parameter conversion]
\label{cor:f124-mixed-affine-bookkeeping}
Suppose the system in Theorem~\ref{thm:f124-mixed-affine-extraction} is
obtained from an $(\ell,w)$-comb, with $I\subseteq[\ell]$, equipped with
a vertex complete to all teeth and anticomplete to all handles, and
suppose fixed $p,q>0$ satisfy
\[
                         m\geq\ell^p,
              \qquad    W\geq w/\ell^q.
\]
Then it gives the complete-or-anticomplete blockade outcome required in
the quantitative comb condition of
Definition~\ref{def:comb-hypothesis}: its
length $k$ and
width satisfy
\begin{equation}
              k\geq\ell^{p\delta},
       \qquad \text{width}\geq \frac{w}{k^{(q+1)/(p\delta)}}.
 \label{eq:mixed-conversion-bounds}
\end{equation}
\end{corollary}

\begin{proof}
Theorem~\ref{thm:f124-mixed-affine-extraction} gives
$k\geq m^\delta\geq\ell^{p\delta}$ and width at least
$W/m\geq w/\ell^{q+1}$, since $m\leq\ell$.  The first inequality in
\eqref{eq:mixed-conversion-bounds} gives
$\ell^{q+1}\leq k^{(q+1)/(p\delta)}$, proving the second.
\end{proof}

This completes the analysis of a family in which every interface is
mixed and additively separable over $\mathbb F_2$.

\section{The auxiliary graph recording mixed interfaces}
\label{sec:outer-graph}

Let $T=\overline F$.  Work in a $T$-free graph containing an
$(\ell,w)$-comb equipped with a vertex $r$ as in
Section~\ref{sec:preliminaries}.  Let $X_i\subseteq B_i$ be
anticonnected for every retained index, and suppose that every pairwise
interface is additively separable over $\mathbb F_2$:
\begin{equation}
 M_{ij}(x,y)=\lambda_i^j(x)\oplus\lambda_j^i(y)\oplus\gamma_{ij}.
 \label{eq:outer-affine}
\end{equation}

\subsection{Local laws for the auxiliary graph}

Let $M$ be the graph on the retained tooth indices in which $ij$ is an
edge exactly when the interface $X_i$--$X_j$ is mixed.  A mixed
interface is \emph{two-sided} if both endpoint functions in
\eqref{eq:outer-affine} are nonconstant.  Otherwise it is
\emph{one-sided}, oriented from its
nonconstant endpoint to its constant endpoint.  A nonedge of $M$ is
labelled $+$ if its interface is complete and $-$ if its interface is
anticomplete.
When capital letters such as $A,B,C$ denote vertices of $M$, the
corresponding tooth and handle are denoted by $X_A$ and $h_A$.

We repeatedly use Lemma~\ref{lem:f124-one-edge-rectangle}: if
$x,x'\in B_i$ are nonadjacent and adjacent tooth vertices $y,z$ lie outside
$B_i$, then exactly one edge in
$\{x,x'\}\times\{y,z\}$ would imply, after relabelling, that
\begin{equation}
       (a,b,c,d,e,f)\longmapsto(x,x',y,h_i,r,z)
 \label{eq:outer-one-edge-witness}
\end{equation}
is an induced $T$.

\begin{lemma}[Propagation of anticompleteness]
\label{lem:f124-anti-propagation}
Suppose $ij\in E(M)$ and the $X_i$ endpoint function of the $ij$
interface is nonconstant.  If $X_i$ is anticomplete to $X_k$, then
$X_j$ is anticomplete to $X_k$.
\end{lemma}

\begin{proof}
Choose a nonedge $xx'$ of $X_i$ crossing the nonconstant $ij$ cut.
For every $y\in X_j$, its column on $\{x,x'\}$ is a singleton.  If
there were an edge $yz$ with $z\in X_k$, then the $z$-column would be
$00$, and \eqref{eq:outer-one-edge-witness} would apply.  Hence no such
edge exists.
\end{proof}

In particular, along a one-sided arc $i\to j$, every tooth
anticomplete to $X_i$ is also anticomplete to $X_j$.  Across a
two-sided edge, the endpoints have the same anticomplete interfaces
outside the pair.

\begin{lemma}[Exact induced-$P_3$ laws]
\label{lem:f124-outer-p3-laws}
Let $A-B-C$ be an induced $P_3$ of $M$.
\begin{enumerate}
 \item If the pure $A$--$C$ interface is anticomplete, then both mixed
       interfaces are one-sided and are oriented $B\to A$ and $B\to C$.
 \item If one incident interface is two-sided and the other is one-sided,
       then $A$--$C$ is complete and the one-sided edge is oriented
       outwards from the centre $B$.
 \item If both incident interfaces are one-sided, they cannot form a
       directed path through $B$.  Thus they are both directed into
       $B$ or both directed out of $B$; the anticomplete endpoint sign
       is possible only in the latter case.
 \item If both incident interfaces are two-sided, then $A$--$C$ is complete,
       and the two endpoint cuts on $X_B$ agree up to complementation.
\end{enumerate}
\end{lemma}

\begin{proof}
For part~(1), if the $A$ endpoint of $AB$ were nonconstant, a crossing
nonedge in $X_A$, an edge of the mixed $B$--$C$ interface, and the
anticomplete $A$--$C$ interface would give
\eqref{eq:outer-one-edge-witness}.  Thus $AB$ is
one-sided from $B$ to $A$; symmetrically it is $B\to C$.

For part~(2), suppose $AB$ is two-sided and the one-sided edge $BC$ is oriented
$C\to B$.  Take a nonedge $bb'$ crossing the $B$-endpoint cut toward
$A$.  A vertex of $X_C$ can be chosen with column $00$ on
$\{b,b'\}$, while every $X_A$-column is a singleton.  Since $A$--$C$
cannot be anticomplete by part~(1), it is complete, and
\eqref{eq:outer-one-edge-witness} is obtained.
The other placement is symmetric.

For part~(3), consider $A\to B\to C$.  A nonedge of $X_B$ crossing the
$B\to C$ cut has a singleton $X_C$-column, while the $A\to B$
interface supplies an $X_A$-column equal to $00$.  The endpoint
interface cannot be anticomplete by part~(1), and if it is complete these
two outside vertices are adjacent, giving
\eqref{eq:outer-one-edge-witness}.  The reverse directed
path is symmetric.

For part~(4), anticompleteness is excluded by part~(1).  If a nonedge $bb'$ of
$X_B$ crosses the cut toward $A$ but not the cut toward $C$, then a
vertex of $X_C$ has column $00$ on $\{b,b'\}$ (the $C$ endpoint
function is nonconstant), while every vertex of $X_A$ gives a singleton
column.  Completeness of $A$--$C$ gives
\eqref{eq:outer-one-edge-witness}.  Interchanging $A,C$
proves the reverse implication.  Since $\overline{G[X_B]}$ is connected,
the two cuts differ by a constant.
\end{proof}

\begin{lemma}[Boundary refinement]
\label{lem:f124-boundary-refinement}
Let $E-A-B$ be an induced $P_3$ of $M$, with $A\to E$ one-sided and
$AB$ two-sided.  Then $E$--$B$ is complete, and every nonedge of $X_A$
which crosses the $A\to E$ cut also crosses the $A$-endpoint two-sided cut
toward $B$.
\end{lemma}

\begin{proof}
The sign and orientation follow from
Lemma~\ref{lem:f124-outer-p3-laws}.  Let $xx'$ cross the first cut but
not the second.  Since the $B$ endpoint function on the two-sided interface
is nonconstant, choose $b\in X_B$ whose column on $\{x,x'\}$ is $00$.
Every $e\in X_E$ has a singleton column, and $eb$ is an edge.  This is
the forbidden rectangle \eqref{eq:outer-one-edge-witness}.
\end{proof}

\begin{lemma}[At most one two-sided edge on an induced $P_4$ of $M$]
\label{lem:f124-outer-p4-one-two-sided}
Every induced $P_4$ of $M$ contains at most one two-sided interface.
\end{lemma}

\begin{proof}
Write the path as $A-B-C-D$.  If $AB$ and $CD$ are two-sided while $BC$ is
one-sided, Lemma~\ref{lem:f124-outer-p3-laws} forces $BC$ to be oriented
$B\to C$ at the first centre and $C\to B$ at the second, impossible.

Suppose next that $AB,BC$ are two-sided.  The $A$--$C$ interface is complete
and the two cuts on $X_B$ synchronise.  If $CD$ is one-sided, the local
law gives $C\to D$ and makes $B$--$D$ complete.  The $A$--$D$ interface
is also complete: if it were anticomplete, propagation across the
$A$-endpoint, whose label function on $AB$ is nonconstant, would make
$B$--$D$ anticomplete.

Choose a nonedge $zz'$ of $X_C$ crossing the $C\to D$ cut.  It must
also cross the $C$-endpoint two-sided cut toward $B$: otherwise the varying
$B$ endpoint supplies a $00$ column, while a $D$-column is a singleton
and $B$--$D$ is complete, giving
\eqref{eq:outer-one-edge-witness}.  Relabel so that $z$ is
anticomplete to $X_D$.

Choose a nonedge $bb'$ of $X_B$ crossing its synchronised cuts toward
$A,C$.  For the fixed $z$, its column on $\{b,b'\}$ is a singleton.
Varying the $A$-endpoint label, choose $x\in X_A$ with the same
singleton column.  For any $d\in X_D$ the map
\begin{equation}
 (a,b,c,d,e,f)\longmapsto(d,h_B,x,b,b',z)
 \label{eq:outer-p4-witness}
\end{equation}
has required edges
\[
 dx,db,db',h_Bb,h_Bb',xb',b'z,xz
\]
and required nonedges
\[
 dh_B,dz,h_Bx,h_Bz,xb,bz,bb'.
\]
It is an induced $T$, a contradiction.

If $CD$ is two-sided, Lemma~\ref{lem:f124-outer-p3-laws}, applied to
$B-C-D$, first shows that $B$--$D$ is complete.  The $A$--$D$
interface is then also complete: if it were anticomplete, propagation
across the nonconstant $A$ endpoint of the two-sided interface $AB$
would make $B$--$D$ anticomplete, a contradiction.  After choosing the
synchronised $bb'$ and matching singleton columns $x,z$, choose
$d\in X_D$ nonadjacent to $z$ using the nonconstant $D$ endpoint of
$CD$.  The same map \eqref{eq:outer-p4-witness} applies.  This excludes two
consecutive two-sided edges and,
by symmetry, completes all cases.
\end{proof}

\subsection{Holes in \texorpdfstring{$M$}{M}}

\begin{theorem}[Every hole in \texorpdfstring{$M$}{M} has even length]
\label{thm:f124-outer-no-odd-hole}
Every induced hole of $M$ has even length.
\end{theorem}

\begin{proof}
Suppose for a contradiction that $M$ has an odd hole.  Its length is at
least five.  First suppose this odd hole contains a two-sided edge $AB$, and
write five
consecutive cycle vertices as
\[
                         D-E-A-B-C.
\]
By Lemma~\ref{lem:f124-outer-p4-one-two-sided}, $EA,BC,DE$ are one-sided.
The two boundary laws orient $A\to E$ and $B\to C$.  At the induced
$P_3$ $D-E-A$, the edge $EA$ enters its centre; hence the other
one-sided edge is $D\to E$, and the pure $D$--$A$ interface is complete.
The $D$--$B$ interface is also complete: if it were anticomplete,
propagation from the $B$-endpoint, whose label function on $BA$ is
nonconstant, would make
$D$--$A$ anticomplete.

Choose a nonedge $xx'$ of $X_A$ crossing the $A\to E$ cut.  Boundary
refinement says that it also crosses the two-sided cut toward $B$.  Choose
$e\in X_E$ and, using the varying $B$ endpoint function, choose
$b\in X_B$ so that their singleton columns on $\{x,x'\}$ hit the same
row, say $x'$.  Since $D\to E$, choose $d\in X_D$ anticomplete to
$X_E$.  The map
\[
 (a,b,c,d,e,f)\longmapsto(d,h_A,b,x,x',e)
\]
has required edges
\[
 db,dx,dx',h_Ax,h_Ax',bx',be,x'e
\]
and required nonedges
\[
 dh_A,de,h_Ab,h_Ae,bx,xx',xe.
\]
Thus it induces $T$, a contradiction.  No edge of this odd hole is two-sided.

All its edges are therefore one-sided.  At every cycle vertex the two
arrows are either both inward or both outward: a directed path is
excluded by Lemma~\ref{lem:f124-outer-p3-laws}, and an anticomplete
distance-two interface forces the outward alternative.  Hence sources
and sinks alternate around the cycle.  The cycle length is even, contrary
to its choice.
\end{proof}

By Theorem~\ref{thm:C5-EH}, there is a constant $\tau>0$ such that
every $C_5$-free graph $Q$ has
\begin{equation}
 \max\{\alpha(Q),\omega(Q)\}\geq |Q|^\tau.
 \label{eq:C5-EH-use}
\end{equation}

\begin{corollary}[Polynomial all-mixed or all-pure set]
\label{cor:f124-outer-c5-extraction}
There is a constant $\tau>0$ such that every auxiliary graph $M$
on $m$ indices contains a set $J$ of at least $m^\tau$ indices for which
either every pair of interfaces is mixed or every pair is pure.
\end{corollary}

\begin{proof}
Theorem~\ref{thm:f124-outer-no-odd-hole} says in particular that $M$ is
$C_5$-free.  Apply Theorem~\ref{thm:C5-EH}.  A clique of $M$ is an all-mixed index set, and
a stable set of $M$ is an all-pure index set, by the definition of $M$.
\end{proof}

\begin{theorem}[Extraction from pure and mixed interfaces]
\label{thm:f124-affine-interface-extraction}
There exist an integer $m_0$ and constants $\rho,\sigma>0$ such that the
following holds.  Let $I$ be a
set of $m\geq m_0$ handles, and let the anticonnected sets
$X_i\subseteq B_i$ have order at least $W$.  Suppose that every
interface has the form \eqref{eq:outer-affine}, and that $h_ih_j$ is a
nonedge whenever the $X_i$--$X_j$ interface is mixed.  Then one of the following
holds:
\begin{enumerate}
 \item there is a pure blockade of length at least $m^\rho$ and width
       at least $W$;
 \item there is a complete or anticomplete blockade of length at least
       $m^\sigma$ and width at least $W/m$.
\end{enumerate}
In either outcome, the blocks lie in pairwise distinct sets $X_i$; in
particular, the actual length of the blockade is at most $m$.
\end{theorem}

\begin{proof}
Choose $m_0$ so that $m_0^\tau\geq4$.
Apply Corollary~\ref{cor:f124-outer-c5-extraction}, and put $s=|J|$.
If $J$ is stable in $M$, the blocks $(X_i:i\in J)$ are pairwise pure,
so item~1 holds with $\rho=\tau$.

If $J$ is a clique in $M$, every pair on $J$ is mixed.  The handles indexed by $J$ are pairwise nonadjacent, by the hypothesis
on mixed interfaces.  Apply
Theorem~\ref{thm:f124-mixed-affine-extraction}.  For its fixed constant
$\delta>0$, it gives a complete or anticomplete blockade of length at
least
\[
                         s^\delta\geq m^{\tau\delta}
\]
and width at least $W/s\geq W/m$.  Thus item~2 holds with
$\sigma=\tau\delta$.
\end{proof}

\section{Multipartite aggregation for clique handles}
\label{sec:multipartite}

\begin{lemma}[Multipartite aggregation]
\label{lem:f124-no-two-edge-aggregation}
Let $m\geq3$, let $V_1,\ldots,V_m$ be disjoint sets of the same integer
order $W$, and let $R$ be a graph on their union with no edges inside a
part.  Suppose that
\[
 xy,yz\in E(R)\quad\Longrightarrow\quad xz\in E(R)
 \tag{A}\label{eq:multipartite-transitivity}
\]
whenever $x,y,z$ belong to three different parts.  Define $Q$ on the
same vertex set by giving it no edges inside a part and making a
cross-part pair adjacent in $Q$ exactly when it is nonadjacent in $R$.

There is a pure pair in $Q$, with its two sides in different parts and
both sides of order at least $W/(4m)$.  If $m\geq144$, there is moreover
a complete or anticomplete blockade in $Q$ of length at least
$\sqrt m/2$ and width at least $W/(2m)$.
\end{lemma}

\begin{proof}
First record the structure of a component $D$ of $R$.  Different
components are complete to one another in $Q$ on every cross-part pair.
If $D$ meets at least three parts, each pair of slices of $D$ in
different parts is anticomplete in $Q$.

For the second assertion, take $x,y\in D$ in different parts and a
shortest $R$-path $P=p_0\cdots p_t$ between them.  Consecutive vertices
of $P$ lie in different parts, since $R$ has no edges inside a part.  If
$p_{q-1},p_q,p_{q+1}$ lay in three different parts,
\eqref{eq:multipartite-transitivity} would give the
shortcut $p_{q-1}p_{q+1}$.  Hence $p_{q-1}$ and $p_{q+1}$ lie in the
same part for every $q$, so the part sequence of $P$ alternates.  Since
its ends lie in different parts, those two alternating parts are exactly
the parts of $x$ and $y$.

Since $D$ meets a third part, take a shortest path from a vertex outside
those two parts to $P$, with its last vertex on $P$, and let $z$ be the
last vertex of this path outside the two parts.  The suffix after $z$
alternates between the two parts.  If its vertices are
$z=u_0,u_1,\ldots,u_r\in V(P)$, induction using
\eqref{eq:multipartite-transitivity} gives
$zu_j\in E(R)$ for every $j$: the base case is $zu_1$, and the three
vertices $z,u_j,u_{j+1}$ lie in distinct parts.  Starting at $u_r$ and
applying the same induction in both directions along $P$ shows that $z$
is adjacent to every vertex of $P$, in particular to $x$ and $y$.  The
three vertices $x,z,y$ lie in distinct parts, so one final application
of \eqref{eq:multipartite-transitivity} gives $xy\in E(R)$.  Thus every cross-part pair in $D$ is an
edge of $R$, proving the asserted sign in $Q$.

Put $s=W/(2m)$ and call a part heavy if one of its component slices has
order at least $s$.  Suppose first that at least $m/2$ parts are heavy,
and select one such slice in every heavy part.  Put
$t=\lceil\sqrt m\rceil$.  If one selected component occurs in at least
$t$ parts, then, for $m\geq9$, its slices give an anticomplete
$t$-blockade in $Q$.  Otherwise more than $\sqrt m/2$ distinct
components occur, because $t-1<\sqrt m$ and their number is at least
\[
             \frac{m}{2(t-1)}>\frac{\sqrt m}{2}.
\]
One selected slice from each gives a complete blockade in $Q$.

Suppose instead that more than $m/2$ parts are light.  Give each
component of $R$, independently and uniformly, one of the labels
$1,\ldots,m$.  In a light part $V_i$, let $X_i$ be the total size of the
slices whose component received label $i$.  Then
\[
 \mathbb E X_i=W/m=2s,
 \qquad
 \operatorname{Var}(X_i)
 \leq \frac{1}{m}\sum_D|D\cap V_i|^2
 <\frac{sW}{m}=\frac{1}{2}(\mathbb E X_i)^2.
\]
Paley--Zygmund therefore gives
\[
 \Pr(X_i\geq s)\geq
 \frac{1}{4}\frac{(\mathbb E X_i)^2}{\mathbb E X_i^2}\geq\frac{1}{6}.
\]
Some labelling makes at least $m/12$ light parts good.  In a good part,
take the union counted by $X_i$.  Blocks belonging to different labels
use different components and hence are pairwise complete in $Q$.  For
$m\geq144$, $m/12\geq\sqrt m$, proving the long conclusion in every
case.

It remains to obtain a pair for arbitrary $m\geq3$.  Suppose first that
some component $D$ has a slice of order at least $s$.  If $D$ meets at
least three parts and another slice has order at least $s$, those two
slices are anticomplete in $Q$.  If all its other slices have order less
than $s$, the complement of $D$ in any other part has order greater than
$W-s$ and is complete in $Q$ to the first slice.  If $D$ meets at most
two parts, a third part is wholly outside $D$ and is complete in $Q$ to
the first slice.  This gives a pure pair of width at least $s$.

Finally suppose every component slice has order less than
$s\leq W/6$.  Fix two parts and independently put every component into
one of two classes with equal probabilities.  In the first part take the
union of class-zero slices, and in the second take the union of class-one
slices.  Each order has mean $W/2$ and variance at most $W^2/24$.
Cantelli's inequality bounds, for each specified order, the probability
that it is below $W/12$ by $6/31$; the union bound is less than one.  Some
assignment therefore makes both orders at least
$W/12\geq W/(4m)$.  They use disjoint component classes and hence are
complete to one another in $Q$.
\end{proof}

In the clique-handle application, $R$ consists of the actual cross-tooth
edges.  Thus a complete blockade in $Q$ is anticomplete in the host and
an anticomplete blockade in $Q$ is complete in the host; either output is
the required complete or anticomplete blockade.

\section{Quantitative reductions}
\label{sec:hjz-quantitative}

Following the five-stage strategy of Huang--Ju--Zhou, we prove a
quantitative extension of their implication from property $(*)$ to
generalized niceness
\cite[Lemma~1.13 and Section~4.2]{HuangJuZhou2026}.  The pure outcome
below retains only a fixed positive power of the number of teeth, so the
exponents must be adjusted; we give the complete adapted argument.
Some differences below may reflect conventions or intended
normalisations that we have overlooked.  We therefore state only the
forms needed for $F$ and prove those variants directly.

\begin{definition}[Quantitative comb condition]
\label{def:comb-hypothesis}
A finite graph family $\mathcal F$ satisfies the
\emph{quantitative comb condition} if there are constants
\[
                       a,b,g,p,q>0
\]
such that every $(\ell,w)$-comb in every
$\overline{\mathcal F}$-free graph $G$, with $\ell,w\geq4$, equipped with
a vertex $r$ satisfying \eqref{eq:root-outside}, has one of the
following outcomes:
\begin{enumerate}
\item[(H)] $\operatorname{hom}(G)\geq w^a$;
\item[(U)] a complete or anticomplete $(k,w/k^b)$-blockade for some
           $k\geq\ell^g$;
\item[(P)] a pure $(k,w/\ell^q)$-blockade for some
           $\ell^p\leq k\leq\ell$.
\end{enumerate}
After replacing $p$ by $\min\{p,1\}$, we may and shall assume that
$0<p\leq1$.
\end{definition}

We first state and prove the leaf-expansion form used in the
reduction.  For graphs $Q,L$, let $\operatorname{ind}_Q(L)$ denote the
number of induced copies of $Q$ in $L$, counted as injections.  The
reduction needs the first outcome relative to the order of the host
graph, as in the later applications of
\cite[Lemma~2.7]{HuangJuZhou2026}.  An implicit normalisation may have
been intended in the displayed statement.  We record and prove the
unordered finite-family version needed here; compare also the related
ordered-graph estimate for sparse pairs in
\cite[Lemma~5.1]{NguyenScottSeymourDensityIV}.

\begin{lemma}[Relative leaf expansion]
\label{lem:hjz-leaf-expansion}
If $\mathcal F$ is leaf-reducible, then there are constants $d>0$ and
$h\geq1$ with the following property.  Let $0<y<1$ and $s>1$, and let
$J$ be an $\mathcal F$-free graph of maximum degree at most $y|J|$.
Then at least one of the following holds:
\begin{enumerate}
\item there are disjoint $X,Y\subseteq V(J)$ such that
\[
 |X|\geq y^{sd+1}|J|,\qquad
 |Y|\geq(1-hy)|J|,
\]
and $Y$ is anticomplete to $X$;
\item $J$ has a $y^s$-restricted induced subgraph of order at least
      $y^{sd+1}|J|$.
\end{enumerate}
\end{lemma}

\begin{proof}
Choose $H\in\mathcal F$ and a leaf $v\in V(H)$ witnessing
leaf-reducibility, and put
\[
 \mathcal F_0=(\mathcal F\setminus\{H\})\cup\{H-v\}.
\]
Write $m=|H|$, and let $u$ be the unique neighbour of $v$ in $H$.
If $m=2$, take $d=h=1$.  Every $\mathcal F$-free graph is then
$K_2$-free and hence edgeless, so the whole graph gives outcome~(2)
for every $0<y<1$ and $s>1$.  We may therefore assume that $m\geq3$.

By leaf-reducibility, $\mathcal F_0$ has the Erd\H{o}s--Hajnal
property.  The standard finite-family extension of the
Erd\H{o}s--Hajnal/viral equivalence
\cite[Theorem~4]{BucicFoxPham2024}, also stated in
\cite[Theorem~1.3]{HuangJuZhou2026}, therefore gives, after increasing the exponent
if necessary to absorb the injection-count convention and the
maximum-degree formulation of restrictedness, a constant $r\geq1$
with the following property.  If $0<\varepsilon<1/2$ and a graph $L$
satisfies
\begin{equation}
 \operatorname{ind}_Q(L)<
       (\varepsilon^r|L|)^{|Q|}
       \qquad\text{for every }Q\in\mathcal F_0,
\label{eq:leaf-viral-form}
\end{equation}
then $L$ has an $\varepsilon$-restricted induced subgraph of order at
least $\varepsilon^r|L|$.  Put
\[
                   r_0=r+1,\qquad d=r_0(m-1)+2,\qquad h=2m.
\]

Let $0<y<1$, let $s>1$, and let $J$ satisfy the hypotheses, and put
$n=|J|$.  If $y\geq1/(2m)$, take
$X=V(J)$ and $Y=\varnothing$.  Then
\[
 |X|=n\geq y^{sd+1}n,\qquad |Y|=0\geq(1-hy)n,
\]
so outcome~(1) holds.  We may consequently assume that
$0<y<1/(2m)$.

Put
\[
                              A=sd+1.
\]
Thus $A>2$.  If $yn<1$, then the integer $\Delta(J)$ is zero, so again
$J$ itself gives outcome~(2).  Suppose that $yn\geq1$.  If
$y^An\leq1$, choose a vertex $z\in V(J)$ and put
\[
 X=\{z\},\qquad
 Y=V(J)\setminus(N_J(z)\cup\{z\}).
\]
The sets $X,Y$ are anticomplete, and
\[
 |X|=1\geq y^An,\qquad
 |Y|\geq n-1-yn\geq(1-my)n\geq(1-hy)n.
\]
Here the penultimate inequality follows from $yn\geq1$ and $m\geq3$.
Thus outcome~(1) holds.  We may assume from now on that
\begin{equation}
                              y^An>1.
\label{eq:leaf-nontrivial}
\end{equation}

Choose $S\subseteq V(J)$ with $|S|=\lceil yn\rceil$, and put
$H'=H-v$.  Suppose first that
\begin{equation}
 \operatorname{ind}_{H'}(J[S])
       \leq y^{A-2}|S|^{m-1}.
\label{eq:leaf-few-copies}
\end{equation}
Since
\[
 A-2=sd-1>s r(m-1),
\]
and $0<y<1$, \eqref{eq:leaf-few-copies} implies
\[
 \operatorname{ind}_{H'}(J[S])
       <((y^s)^r|S|)^{m-1}.
\]
Every member of $\mathcal F_0$ other than $H'$ belongs to
$\mathcal F\setminus\{H\}$ and hence has no induced copy in $J[S]$.
Moreover $y^s<1/2$.  Applying \eqref{eq:leaf-viral-form} with
$\varepsilon=y^s$, we obtain a $y^s$-restricted induced subgraph $R$
of $J[S]$ such that
\[
 |R|\geq y^{sr}|S|
      \geq y^{sr+1}n
      \geq y^{sd+1}n.
\]
This is outcome~(2).

It remains to consider the case
\begin{equation}
 \operatorname{ind}_{H'}(J[S])
       >y^{A-2}|S|^{m-1}.
\label{eq:leaf-many-copies}
\end{equation}
Let $K=H-\{u,v\}$.  For each induced embedding $\phi$ of $K$ in
$J[S]$, let $I_\phi$ be the set of induced embeddings of $H'$ in
$J[S]$ whose restriction to $K$ is $\phi$.  There are at most
$|S|^{m-2}$ choices for $\phi$.  If $|I_\phi|<y^An$ for every $\phi$,
then, since $n\leq|S|/y$,
\begin{align*}
 \operatorname{ind}_{H'}(J[S])
   &=\sum_\phi |I_\phi|\\
   &<|S|^{m-2}y^An\\
   &\leq y^{A-1}|S|^{m-1}
    <y^{A-2}|S|^{m-1},
\end{align*}
contrary to \eqref{eq:leaf-many-copies}.  Hence there is an embedding
$\phi$ for which $|I_\phi|\geq y^An$.

Put $P=\phi(V(K))$, and let
\[
                    X=\{\psi(u):\psi\in I_\phi\}.
\]
For fixed $\phi$, an extension $\psi$ is determined by $\psi(u)$, and
therefore $|X|=|I_\phi|\geq y^An$.  Let
\[
 Y=\{z\in V(J)\setminus S:N_J(z)\cap P=\varnothing\}.
\]
Since $|P|=m-2$, $\Delta(J)\leq yn$, and
$|S|=\lceil yn\rceil$, we have
\[
 |Y|\geq n-\lceil yn\rceil-(m-2)yn
     \geq(1-(m-1)y)n-1
     \geq(1-my)n
     \geq(1-hy)n.
\]
The penultimate inequality follows from
\eqref{eq:leaf-nontrivial}, which implies $yn>1$.

Finally, $X$ is anticomplete to $Y$.  Indeed, if $x_0\in X$ and
$y_0\in Y$ were adjacent, choose $\psi\in I_\phi$ with
$\psi(u)=x_0$.  The map that agrees with $\psi$ on $H'$ and sends $v$
to $y_0$ would be an induced embedding of $H$ in $J$: the vertex
$y_0$ is adjacent to $x_0=\psi(u)$ and has no neighbours in
$P=\psi(V(H')\setminus\{u\})$.  This contradicts that $J$ is
$\mathcal F$-free.  Thus
\[
 |X|\geq y^{sd+1}n,\qquad |Y|\geq(1-hy)n,
\]
and outcome~(1) holds.
\end{proof}

\begin{lemma}[Sparse-graph comb lemma
{\cite[Lemma~2.10]{HuangJuZhou2026}}]
\label{lem:hjz-sparse-comb}
Let $0<x\leq y\leq2^{-8}$, and let $J$ be a graph of maximum degree at
most $y^3|J|$, where $|J|\geq y^{-4}$.  At least one of the following
holds:
\begin{enumerate}
\item there are disjoint $X,Y\subseteq V(J)$ such that
\[
 |X|\geq y^4|J|,\qquad |Y|\geq(1-4y)|J|,
\]
and $Y$ is $x$-sparse to $X$;
\item $J$ is $2y^4$-restricted;
\item for some integer $\ell\in[y^{-1},x^{-2}]$, the graph $J$
contains an
\[
               \left(\ell,\frac{y^4|J|}{\ell^2}\right)\text{-comb}
\]
equipped with a vertex satisfying \eqref{eq:root-outside}.
\end{enumerate}
\end{lemma}

We shall also use the following conversion.

\begin{lemma}[Pure-or-sparse conversion
{\cite[Lemma~2.5]{HuangJuZhou2026}}]
\label{lem:hjz-pure-sparse-conversion}
Let $0<\varepsilon<1/2$, let $D\geq1$, and put
$x=\varepsilon^{5D}$.  Suppose that
$|J|\geq\varepsilon^{-10D^2}$ and that, for every induced subgraph
$L$ of $J$ with $|L|\geq\varepsilon^D|J|$, there is a pure or
$x$-sparse
\[
                    (k,|L|/k^D)\text{-blockade},
             \qquad 2\leq k\leq x^{-1}.
\]
Then $J$ has an
\[
       (\varepsilon^{-1},x^{2D}|J|)\text{-blockade}
\]
whose blocks are pairwise complete or weakly
$\varepsilon^D$-sparse.
\end{lemma}

The next statement is the parameter range of
\cite[Lemma~2.8]{HuangJuZhou2026} used here.  We include a direct proof
of the product bound because the required constant was not clear to us
from the displayed comparison; we may have missed an intended
adjustment.

\begin{lemma}[Blockade iteration]
\label{lem:hjz-blockade-iteration}
Let $B\geq4$, put $\zeta=2^{-4B}$, and let
$0<x<1$, $0<y\leq\zeta$, and $A>1$.  Suppose that
$|J|\geq y^{-(A+2)}$ and that every induced subgraph $L$ of $J$ with
$|L|\geq\zeta|J|$ contains disjoint sets $X,Y$ satisfying
\[
 |X|\geq y^A|L|,\qquad
 |Y|\geq(1-By)|L|,
\]
where $Y$ is either $x$-sparse or complete to $X$.  Then $J$ has an
$x$-sparse or complete
\[
                     (y^{-1},y^{A+2}|J|)\text{-blockade}.
\]
\end{lemma}

\begin{proof}
Put $u=By$.  Since $B\geq4$ and $y\leq2^{-4B}$, we have $u<1/4$.
Furthermore,
\[
 \log(1-u)\geq-\frac{u}{1-u}\geq-2(\log 2)u,
\]
because $1/(1-u)<4/3<2\log 2$.  Hence
\begin{equation}
        (1-By)^{2/y}\geq
        \exp(-4B\log 2)=2^{-4B}=\zeta.
\label{eq:hjz-product-bound}
\end{equation}

Choose a sequence $(C_1,\ldots,C_m)$ of pairwise disjoint vertex sets
of maximum possible length such
that
\[
 |C_i|\geq y^{A+2}|J|\quad(i\in[m]),\qquad
 |C_m|\geq(1-By)^m|J|,
\]
and, for each $i<m$, every later block is either $x$-sparse or
complete to $C_i$, with the choice fixed for that $i$.  Such a sequence
exists with $m=1$, since $|J|\geq y^{-(A+2)}$.

Suppose that $m<2/y$.  By
\eqref{eq:hjz-product-bound},
\[
                   |C_m|\geq\zeta|J|.
\]
The hypothesis applied to $J[C_m]$ gives disjoint $X,Y\subseteq C_m$.
Since $y\leq\zeta\leq1$, we have $\zeta\geq y^2$, and consequently
\[
       |X|\geq y^A|C_m|
             \geq y^A\zeta|J|
             \geq y^{A+2}|J|.
\]
Also
\[
       |Y|\geq(1-By)|C_m|
             \geq(1-By)^{m+1}|J|.
\]
In addition, $By<1/4$ and $|C_m|\geq\zeta|J|$, so
\[
 |Y|>\frac{3}{4}\zeta|J|
       \geq y^{A+2}|J|.
\]
For the last inequality, use $A>1$ and
$y\leq\zeta\leq1/2$, which give
$y^{A+2}\leq y^3\leq\zeta/4$.
Both $X$ and $Y$ inherit the fixed relation of $C_m$ to every earlier
block, while $Y$ is $x$-sparse or complete to $X$.  Thus
$(C_1,\ldots,C_{m-1},X,Y)$ contradicts the maximality of $m$.
Therefore $m\geq2/y$.

For each index $i$, classify it according to whether all later blocks
are $x$-sparse or complete to $C_i$, assigning the last index
arbitrarily.  One class has at least $m/2\geq y^{-1}$ indices.
The corresponding subsequence is the required $x$-sparse or complete
blockade.
\end{proof}

The statement of~\cite[Lemma~2.9]{HuangJuZhou2026} treats sparse graphs,
but our application is complement-symmetric.  Perhaps complementation
was intended implicitly.  We therefore prove the restricted-graph
analogue needed below.  This is a variant, not a restatement, of that
lemma.

\begin{lemma}[Restricted-graph iteration]
\label{lem:hjz-restricted-iteration}
Let $0<x<c<1$, let $b_1>1$ and $b_2,b_3>0$, and suppose that
\[
                         b_1b_2\geq b_2+b_3.
\]
Suppose that $J$ contains a $c$-restricted induced subgraph of order at
least $c^{b_2}|J|$, and that, for every $\lambda\in[x,c]$ and every
$\lambda$-restricted induced subgraph $L$ with
$|L|\geq\lambda^{b_2}|J|$, the graph $L$ contains a
$\lambda^{b_1}$-restricted induced subgraph $L'$ satisfying
\[
                         |L'|\geq\lambda^{b_3}|L|.
\]
Then $J$ contains an $x$-restricted induced subgraph of order at least
$x^{b_1b_2}|J|$.
\end{lemma}

\begin{proof}
Start with the asserted $c$-restricted graph and iterate while the
current parameter $\lambda$ is at least $x$, replacing it by
$\lambda'=\lambda^{b_1}$.  If the current graph has order at least
$\lambda^{b_2}|J|$, then the next graph has order at least
\[
 \lambda^{b_2+b_3}|J|
     \geq\lambda^{b_1b_2}|J|
     =(\lambda')^{b_2}|J|.
\]
Thus the induction invariant is preserved; no fixed choice between a
graph and its complement is needed.  At the first step for which
$\lambda'<x$, the preceding value satisfies $\lambda\geq x$, and hence
$\lambda'=\lambda^{b_1}\geq x^{b_1}$.  The resulting graph is
$x$-restricted and has order at least
\[
               (\lambda')^{b_2}|J|
                    \geq x^{b_1b_2}|J|.
\]
\end{proof}

\subsection{Generalized niceness implies the
Erd\H{o}s--Hajnal property}

We next give a complete proof, under the conventions above, of the
reduction stated in Lemma~1.12 of~\cite{HuangJuZhou2026}, using the
quantitative statements above in precisely the forms in which they are
needed.

\begin{theorem}[Generalized niceness reduction]
\label{thm:hjz-generalized-nice}
Let $\mathcal F$ be a finite family.  If $\mathcal F$ is
leaf-reducible, wonderful, and generalized nice, then $\mathcal F$
has the Erd\H{o}s--Hajnal property.
\end{theorem}

\begin{proof}
It is enough to prove a polynomial clique-or-stable-set bound for every
$\overline{\mathcal F}$-free graph, since complementation preserves
$\operatorname{hom}$.

We use the following three standard results of Nguyen--Scott--Seymour,
stated in~\cite[Lemmas~2.3, 2.4, and~2.6]{HuangJuZhou2026}; see
also~\cite{NguyenScottSeymourP5}.
\begin{enumerate}
\item If every anticomponent of a graph $Q$ has order less than
      $|Q|/k$, where $k\geq2$ is an integer, then $Q$ has a complete
      $(k,|Q|/k^2)$-blockade.
\item If $t\geq5$, $0<\eta\leq1/4$,
      $\ell=\lceil\eta^{-1}\rceil$, and
      $(C_1,\ldots,C_\ell)$ is a blockade of width at least $w_0$,
      then one may choose $D_i\subseteq C_i$ with
      \[
                 |D_i|=\lceil\eta\lceil w_0\rceil\rceil
      \]
      so that every weakly $\eta^t$-sparse pair becomes
      $\eta^{t-5}$-sparse in both directions.
\item If $r\geq1$, $0<\eta<1/2$, and every induced subgraph $L$ of a
      graph $Q$ with $|L|\geq\eta^{2r}|Q|$ contains a complete or
      anticomplete
      \[
                    (k,|L|/k^r)\text{-blockade},
             \qquad 2\leq k\leq\eta^{-1},
      \]
      then $Q$ has an $\eta$-restricted induced subgraph of order at
      least $\eta^{3r}|Q|$.
\end{enumerate}

Let
\[
 \gamma_1\geq3,\quad \gamma_2\geq8,\quad
 \gamma_3,\gamma_4,\gamma_5,\gamma_8>0,\quad
 \gamma_6\geq1,\quad\gamma_7\geq4
\]
witness generalized niceness.  Let $d>0,h\geq1$ be supplied by
Lemma~\ref{lem:hjz-leaf-expansion}, and let $\nu\geq6$ witness
wonderfulness.  Put
\[
 \begin{split}
 A_1&=\nu\gamma_3,\qquad A_2=\gamma_4,\\
 A_3&=\nu(\gamma_1+\gamma_8+5)+\gamma_5+4d+1,\\
 A_4&=4,\qquad A_5=4+h.
 \end{split}
\]
In particular,
\[
 A_3\geq A_4,\quad A_3\geq4d+1,\quad
 A_3\geq\nu\gamma_8,\quad A_3\geq\gamma_5,\quad
 A_3\geq\nu(\gamma_1+5).
\]

\medskip
\noindent\emph{First reduction.}
We claim that for every $0<y<1/2$ and every $y$-restricted
$\overline{\mathcal F}$-free graph $G$, at least one of the following
holds:
\begin{enumerate}[label=(S\arabic*)]
\item $\operatorname{hom}(G)\geq(y^{A_1}|G|)^{A_2}$;
\item $G$ has a $y^{A_4}$-restricted induced subgraph of order at
      least $y^{A_3}|G|$;
\item $G$ has a complete or anticomplete
      $(k,|G|/k^{A_3})$-blockade with $k\geq y^{-1}$;
\item there are disjoint $X,Y\subseteq V(G)$ such that
      \[
       |X|\geq y^{A_3}|G|,\qquad
       |Y|\geq(1-A_5y)|G|,
      \]
      and $Y$ is complete or anticomplete to $X$.
\end{enumerate}

If $|G|\leq y^{-A_3}$, a singleton gives (S2), so assume
$|G|>y^{-A_3}$.  If $\overline G$ has maximum degree at most
$y|G|$, apply Lemma~\ref{lem:hjz-leaf-expansion} to $\overline G$
with exponent $4$.  Since $\overline G$ is $\mathcal F$-free,
complementing its conclusions gives (S2) or (S4), using
$A_3\geq4d+1$ and $A_5\geq h$.

We may therefore suppose that $G$ has maximum degree at most $y|G|$.
Set $\eta=y^\nu$.  Apply generalized niceness to $G$ with parameter
$\eta$.  Its second, third, and fourth outcomes give respectively
(S1), (S3), and (S2), because
\[
 \eta^{\gamma_3}=y^{A_1},\qquad
 \eta^{-\gamma_6}\geq y^{-1},\qquad
 \eta^{\gamma_7}\leq y^{A_4},\qquad
 \eta^{\gamma_8}\geq y^{A_3}.
\]

It remains to treat the first generalized-niceness outcome.  Retain
exactly
\[
             \ell=\lceil\eta^{-1}\rceil
\]
blocks, write them as $(C_1,\ldots,C_\ell)$, and put
$w_0=\eta^{\gamma_1}|G|$.  Since
$A_3\geq\nu(\gamma_1+5)$ and $|G|>y^{-A_3}$,
\[
                  w_0\geq\eta^{-5}>1.
\]
Apply the weak-sparsity cleanup result and obtain equal-sized
$D_i\subseteq C_i$ with
\[
 |D_i|=\lceil\eta\lceil w_0\rceil\rceil\geq\eta w_0,
\]
such that every noncomplete pair is
$\eta^{\gamma_2-5}$-sparse in both directions.

Suppose first that some $D_i$ has no anticomponent of order at least
$|D_i|/\ell$.  The anticomponent result gives a complete
$(\ell,|D_i|/\ell^2)$-blockade.  Since $\eta\leq1/4$ and hence
$\ell\leq\eta^{-2}$,
\[
 \frac{|D_i|}{\ell^2}
   \geq\eta^5w_0
   =\eta^{\gamma_1+5}|G|
   =y^{\nu(\gamma_1+5)}|G|
   \geq y^{A_3}|G|.
\]
Also $\ell\geq\eta^{-1}=y^{-\nu}\geq y^{-1}$, so (S3) holds.

We may thus choose, for every $i$, an anticomponent of $D_i$ of order
at least $|D_i|/\ell$.  A connected graph has a connected induced
subgraph of every smaller positive order: repeatedly delete a leaf
of a spanning tree.  Applying this in the complements, choose
anticonnected $B_i\subseteq D_i$ of the common order
\[
                   w=\left\lceil\frac{|D_i|}{\ell}\right\rceil.
\]
Here
\[
 w\geq\frac{|D_i|}{\ell}
   \geq\eta^2|D_i|
   \geq\eta^3w_0
   \geq\eta^{-2}>1.
\]
Every pair $B_i,B_j$ is complete or
$\eta^{\gamma_2-7}$-sparse, because
$|B_j|\geq\eta^2|D_j|$, and
\[
 \eta^{\gamma_2-7}
       =y^{\nu(\gamma_2-7)}
       \leq y^\nu.
\]
Also $\ell\geq\eta^{-1}=y^{-\nu}$.  Thus wonderfulness applies to the
equal-width blockade
$\mathcal B=(B_1,\ldots,B_\ell)$.

Its first outcome gives a $y^4$-restricted subgraph of order at least
\[
 w\geq\eta^3w_0
      =y^{\nu(\gamma_1+3)}|G|
      \geq y^{A_3}|G|,
\]
and hence (S2).  In its second outcome, choose the corresponding
index $i$, and put
\[
                  U=V(G)\setminus\bigcup_{j=1}^{\ell}B_j.
\]
At most $y|G|$ vertices $v$ of $U$ satisfy
$0<|N_G(v)\cap B_i|<|B_i|/2$.  Since $G$ has maximum degree at most
$y|G|$, edge counting shows that at most $2y|G|$ vertices of $U$ have
at least $|B_i|/2$ neighbours in $B_i$.

Moreover, $|D_i|\leq\lceil w_0\rceil\leq2w_0$, and
$|D_i|/\ell>1$, so
\[
 |B_i|\leq\frac{2|D_i|}{\ell}\leq\frac{4w_0}{\ell}.
\]
Consequently,
\[
 \sum_{j=1}^{\ell}|B_j|
       \leq4w_0
       \leq\eta^{-2}\eta^{\gamma_1}|G|
       \leq\eta|G|
       \leq y|G|.
\]
Let $Y$ be the set of vertices of $U$ anticomplete to $B_i$.  Then
\[
 |Y|\geq |G|-\sum_j|B_j|-3y|G|
       \geq(1-4y)|G|
       \geq(1-A_5y)|G|,
\]
while
\[
 |B_i|\geq\eta^3w_0
        =y^{\nu(\gamma_1+3)}|G|
        \geq y^{A_3}|G|.
\]
Thus $X=B_i$ and $Y$ give (S4), proving the first reduction.

\medskip
\noindent\emph{One-block reduction.}
Put
\[
 \rho=2^{-4A_5},\qquad
 B_1=A_1+1,\quad B_2=A_2,\quad
 B_3=A_3+2,\quad B_4=A_4.
\]
We claim that for $0<y\leq\rho$, every
$\rho y$-restricted $\overline{\mathcal F}$-free graph $G$ has one
of the following:
\begin{enumerate}[label=(T\arabic*)]
\item $\operatorname{hom}(G)\geq(y^{B_1}|G|)^{B_2}$;
\item a complete or anticomplete
      $(k,|G|/k^{B_3})$-blockade with $k\geq y^{-1}$;
\item a $y^{B_4}$-restricted induced subgraph of order at least
      $y^{B_3}|G|$.
\end{enumerate}

If $|G|\leq y^{-B_3}$, a singleton gives (T3).  Otherwise, suppose
first that every induced $L\subseteq G$ with $|L|\geq\rho|G|$ has
the pair in (S4).  The proof of
Lemma~\ref{lem:hjz-blockade-iteration}, with the sparse alternative
replaced throughout by the anticomplete alternative, gives an
anticomplete or complete
\[
             (y^{-1},y^{A_3+2}|G|)\text{-blockade}.
\]
This is (T2).  Notice that the valid product estimate used here is
exactly
\[
                  (1-A_5y)^{2/y}\geq2^{-4A_5}=\rho.
\]

Otherwise choose $L\subseteq G$ with $|L|\geq\rho|G|$ having no such
pair.  Since $y\leq\rho$, we have $|L|\geq y|G|$, and the
$\rho y$-restrictedness of $G$ implies that $L$ is $y$-restricted.
Apply the first reduction to $L$.  Outcome (S4) is excluded.  The
other three outcomes imply (T1)--(T3), using
\[
 \begin{split}
 y^{A_1}|L|&\geq y^{A_1+1}|G|,\\
 y^{A_3}|L|&\geq y^{A_3+1}|G|
                         \geq y^{A_3+2}|G|,\\
 \frac{|L|}{k^{A_3}}
   &\geq\frac{y|G|}{k^{A_3}}
    \geq\frac{|G|}{k^{A_3+1}}
    \geq\frac{|G|}{k^{A_3+2}},
 \end{split}
\]
where the last line uses $k\geq y^{-1}$.  The one-block reduction
follows.

\medskip
\noindent\emph{Iterative reduction.}
Set
\[
                     R=B_1+3B_3,\qquad S=B_2.
\]
We claim that for every $0<x<\rho^2$ and every
$\rho^2$-restricted $\overline{\mathcal F}$-free graph $G$, at least
one of the following holds:
\begin{enumerate}[label=(I\arabic*)]
\item $G$ has an $x$-restricted induced subgraph of order at least
      $x^R|G|$;
\item $\operatorname{hom}(G)\geq(x^R|G|)^S$;
\item $G$ has a complete or anticomplete
      $(k,|G|/k^R)$-blockade with $k\geq2$.
\end{enumerate}

Suppose none holds.  Let $\lambda\in[x,\rho^2]$, put
$y=\lambda/\rho$, and let $L\subseteq G$ be a
$\lambda$-restricted induced subgraph with
\[
                 |L|\geq
          \lambda^{4B_3/B_4}|G|.
\]
Then $0<y\leq\rho$, and $\lambda=\rho y\geq y^2$.  Apply the
  one-block reduction to $L$.  Its clique-or-stable-set outcome would imply
(I2), because
\[
 y^{B_1}|L|
  \geq y^{B_1+8B_3/B_4}|G|
  \geq x^R|G|,
\]
where $8B_3/B_4=2B_3$, $y\geq\lambda\geq x$, and
$R=B_1+3B_3$.  Its blockade outcome would imply (I3),
because $k\geq y^{-1}$ and
\[
 \frac{|L|}{k^{B_3}}
 \geq\frac{y^{8B_3/B_4}|G|}{k^{B_3}}
 \geq\frac{|G|}{k^{B_3+8B_3/B_4}}
\geq\frac{|G|}{k^R}.
\]
Here $y\geq k^{-1}$, $8B_3/B_4=2B_3$, and
$R=B_1+3B_3$ justify the last two comparisons.
Consequently, $L$ contains a
$\lambda^{B_4/2}$-restricted induced subgraph of order at least
$\lambda^{B_3}|L|$: indeed,
\[
                  y^{B_4}\leq\lambda^{B_4/2},
            \qquad y^{B_3}\geq\lambda^{B_3}.
\]

Apply Lemma~\ref{lem:hjz-restricted-iteration} with
\[
 c=\rho^2,\qquad
 b_1=B_4/2,\qquad b_2=4B_3/B_4,\qquad b_3=B_3.
\]
Here
\[
 b_1b_2=2B_3
      =B_3+\frac{4B_3}{B_4}=b_2+b_3,
\]
because $B_4=4$.  Taking $G$ itself gives the initial
$c$-restricted induced subgraph of order at least $c^{b_2}|G|$.
The lemma therefore gives an $x$-restricted induced subgraph of order
at least
\[
                 x^{2B_3}|G|\geq x^R|G|,
\]
contradicting the failure of (I1).

\medskip
\noindent\emph{R\"odl reduction.}
Put $\xi=\rho^2$.  Choose one graph
$H_0\in\overline{\mathcal F}$.  By R\"odl's theorem
\cite{Rodl1986}, applied to $H_0$, there is $\delta>0$ such that every
$\overline{\mathcal F}$-free graph has a $\xi$-restricted induced
subgraph of order at least $\delta$ times its order.  Choose
$T\geq\max\{1,2R\}$ so large that
\[
                  \delta\geq\xi^{T/2},
              \qquad \delta\geq2^{R-T}.
\]
We claim that for every $0<x<1/2$ and every
$\overline{\mathcal F}$-free graph $G$, one of the following holds:
\begin{enumerate}[label=(R\arabic*)]
\item an $x$-restricted induced subgraph of order at least $x^T|G|$;
\item a complete or anticomplete
      $(k,|G|/k^T)$-blockade with $k\geq2$;
\item a clique or stable set of order at least $(x^T|G|)^S$.
\end{enumerate}

Choose a $\xi$-restricted $L\subseteq G$ with
$|L|\geq\delta|G|$.  If $x\geq\xi$, then $L$ is $x$-restricted and
\[
 |L|\geq2^{R-T}|G|
      \geq x^{T-R}|G|
      \geq x^T|G|,
\]
so (R1) holds.  If $x<\xi$, apply the iterative reduction to $L$.
Its first two outcomes give (R1) or (R3), since
\[
 x^R|L|\geq x^R\xi^{T/2}|G|
          \geq x^{R+T/2}|G|
          \geq x^T|G|.
\]
In its blockade outcome, for $k\geq2$,
\[
 \frac{|L|}{k^R}
     \geq\frac{\delta|G|}{k^R}
     \geq\frac{|G|}{k^T},
\]
because $\delta\geq2^{R-T}\geq k^{R-T}$.  Thus (R2) holds.

\medskip
\noindent\emph{Clique-or-stable-set conclusion.}
Let
\[
 q_0=42T^2,\qquad n_0=2^{q_0},\qquad
 \theta=\min\{q_0^{-1},S/2\}.
\]
We prove that every $\overline{\mathcal F}$-free graph $G$ satisfies
\[
                       \operatorname{hom}(G)\geq |G|^\theta.
\]
The assertion is immediate for $|G|\leq n_0$: for
$2\leq|G|\leq n_0$, one has
$|G|^\theta\leq n_0^{1/q_0}=2$, while $|G|=1$ is trivial.

Suppose $n=|G|>n_0$ and, for a contradiction,
$\operatorname{hom}(G)<n^\theta$.  Put
\[
       x=n^{-1/(3T)},\qquad
       \varepsilon=x^{1/(7T)}
                  =n^{-1/(21T^2)}.
\]
Then $x<\varepsilon<1/4$.

We claim that every induced $L\subseteq G$ with
$|L|\geq\varepsilon^{2T}n$ has a complete or anticomplete
\[
             (k,|L|/k^T)\text{-blockade},
             \qquad 2\leq k\leq\varepsilon^{-1}.
\]
Otherwise apply the R\"odl reduction to such an $L$ with parameter
$x$.  In outcome (R1), the restricted subgraph has order $N$ at least
\[
 x^T|L|
   \geq x^{T+2/7}n
   =x^{-2T+2/7}
   \geq x^{-1}.
\]
After complementing if necessary, its maximum degree is at most $xN$.
The greedy bound gives a clique or stable set of order at least
\[
 \frac{N}{xN+1}
   =\frac1{x+1/N}
   \geq\frac1{2x}
   \geq x^{-1/2}
   =n^{1/(6T)}
   \geq n^\theta,
\]
a contradiction.  In outcome (R2), failure of the claim forces
$k>\varepsilon^{-1}$; choosing one vertex from every block gives a
clique or stable set of order at least
\[
 \varepsilon^{-1}
       =n^{1/(21T^2)}
       \geq n^\theta.
\]
Finally, outcome (R3) gives a clique or stable set of order at least
\[
 \begin{split}
 (x^T|L|)^S
  &\geq(x^{T+2/7}n)^S\\
  &=n^{S(1-(T+2/7)/(3T))}
   \geq n^{S/2}
   \geq n^\theta,
 \end{split}
\]
because
\[
 1-\frac{T+2/7}{3T}
       =\frac23-\frac{2}{21T}\geq\frac12.
\]
This proves the claim.

Apply the polynomial R\"odl conversion stated at the beginning of the
proof with $r=T$ and parameter $\varepsilon$.  We obtain an
$\varepsilon$-restricted induced subgraph $Q$ with
\[
 |Q|\geq\varepsilon^{3T}n
      =\varepsilon^{-21T^2+3T}
      \geq\varepsilon^{-1}.
\]
After complementing if necessary, the greedy bound gives a
clique or stable set of order at least
\[
 \frac1{\varepsilon+1/|Q|}
   \geq\frac1{2\varepsilon}
   \geq\varepsilon^{-1/2}
   =n^{1/(42T^2)}
   \geq n^\theta,
\]
the final contradiction.  Hence every
$\overline{\mathcal F}$-free graph has the claimed polynomial
clique-or-stable-set bound, and complementation gives the
Erd\H{o}s--Hajnal property for $\mathcal F$.
\end{proof}

\subsection{From the quantitative comb condition to generalized niceness}

For the remainder of this section, fix a leaf-reducible family
$\mathcal F$ satisfying Definition~\ref{def:comb-hypothesis} with
constants $a,b,g,p,q$.  We now propagate that conclusion through all
five stages, following the strategy in~\cite{HuangJuZhou2026}.  Let $d,h$ be supplied by
Lemma~\ref{lem:hjz-leaf-expansion}, and define
\begin{align}
 A&=\max\{4d+1,4\},\notag\\
 B&=\max\left\{h,4,\frac{1}{3p}\right\},\qquad
 \zeta=2^{-4B},\notag\\
 U_0&=b+\frac{6}{g},&
 U_1&=b+\frac{7}{g},&
 U_2&=b+\frac{27A}{g},\notag\\
 P_0&=\frac{q+6}{p},&
 P_1&=P_0+\frac{1}{p},&
 P_2&=P_1+\frac{10(A+2)}{p},\notag\\
 P_*&=\max\{P_2,29A\},&
 K_*&=\min\{\zeta^{-3p},\zeta^{-1}\}.
\label{eq:hjz-modified-constants}
\end{align}
The choice of $B$ gives
\begin{equation}
                         K_*\geq16.
\label{eq:hjz-kstar}
\end{equation}

\begin{lemma}[First comb reduction]
\label{lem:hjz-modified-first}
Let $0<x\leq y\leq\zeta$, and let $J$ be a $y^3$-restricted
$\overline{\mathcal F}$-free graph.  At least one of the following
holds:
\begin{enumerate}
\item there are disjoint $X,Y\subseteq V(J)$ such that
\[
 |X|\geq y^A|J|,\qquad |Y|\geq(1-By)|J|,
\]
where $Y$ is $x$-sparse or complete to $X$;
\item $J$ has a $2y^4$-restricted induced subgraph of order at least
      $y^A|J|$;
\item $\operatorname{hom}(J)\geq(x^9|J|)^a$;
\item $J$ has a complete or anticomplete
\[
             (k,|J|/k^{U_0})\text{-blockade},
             \qquad k\geq y^{-g};
\]
\item $J$ has a pure
\[
             (k,|J|/k^{P_0})\text{-blockade},
             \qquad k\in[y^{-p},x^{-2}].
\]
\end{enumerate}
\end{lemma}

\begin{proof}
If $\overline J$ has maximum degree at most $y^3|J|$, apply
Lemma~\ref{lem:hjz-leaf-expansion} to $\overline J$ with parameter $y$
and exponent $4$.  Recall that $\overline J$ is $\mathcal F$-free.
After complementation its first outcome gives outcome~(1), and its
second gives outcome~(2), because $A\geq4d+1$ and $B\geq h$.

We may therefore assume that $J$ has maximum degree at most
$y^3|J|$.  If $|J|\leq x^{-9}$, outcome~(3) holds.  Otherwise
\[
                  |J|>x^{-9}\geq y^{-4},
\]
so Lemma~\ref{lem:hjz-sparse-comb} applies.  Its first two outcomes
give outcomes~(1) and~(2), since $A\geq4$ and $B\geq4$.

In the remaining case there are
\[
 \ell\in[y^{-1},x^{-2}],\qquad
 w=\frac{y^4|J|}{\ell^2},
\]
and a comb to which Definition~\ref{def:comb-hypothesis} applies;
indeed, $\ell\geq y^{-1}\geq\zeta^{-1}>4$.  The parameter bounds imply
\begin{equation}
                       w\geq\frac{|J|}{\ell^6},
             \qquad    w\geq x^8|J|,
             \qquad    w\geq x^{-1}\geq4.
\label{eq:hjz-comb-width}
\end{equation}
Outcome (H) gives outcome~(3).  In outcome (U), since
$k\geq\ell^g$,
\[
 \frac{w}{k^b}
       \geq \frac{|J|}{\ell^6k^b}
       \geq \frac{|J|}{k^{b+6/g}}
       =\frac{|J|}{k^{U_0}},
 \qquad k\geq y^{-g}.
\]
In outcome (P), since $k\geq\ell^p$,
\[
 \frac{w}{\ell^q}
       \geq \frac{|J|}{\ell^{q+6}}
       \geq \frac{|J|}{k^{(q+6)/p}}
       =\frac{|J|}{k^{P_0}},
\]
and $k\in[y^{-p},x^{-2}]$.  These are outcomes~(4) and~(5).
\end{proof}

\begin{lemma}[One-block reduction]
\label{lem:hjz-modified-once}
Let $0<x\leq y\leq\zeta$, and let $J$ be a
$\zeta y^3$-restricted $\overline{\mathcal F}$-free graph.  At least
one of the following holds:
\begin{enumerate}
\item $J$ has an $x$-sparse or complete
\[
                 (y^{-1},y^{A+2}|J|)\text{-blockade};
\]
\item $J$ has a $2y^4$-restricted induced subgraph of order at least
      $y^{A+2}|J|$;
\item $\operatorname{hom}(J)\geq(x^{10}|J|)^a$;
\item $J$ has a complete or anticomplete
\[
            (k,|J|/k^{U_1})\text{-blockade},
            \qquad k\geq y^{-g};
\]
\item $J$ has a pure
\[
            (k,|J|/k^{P_1})\text{-blockade},
            \qquad k\in[y^{-p},x^{-2}].
\]
\end{enumerate}
\end{lemma}

\begin{proof}
If $|J|\leq y^{-(A+2)}$, a singleton gives outcome~(2).  We may
therefore assume the reverse inequality.

If every induced subgraph $L$ with $|L|\geq\zeta|J|$ has the pair in
outcome~(1) of Lemma~\ref{lem:hjz-modified-first}, then
Lemma~\ref{lem:hjz-blockade-iteration} gives outcome~(1).  Otherwise
choose an induced $L\subseteq J$, with $|L|\geq\zeta|J|$, having no
such pair.  Since $y\leq\zeta$,
\begin{equation}
                             |L|\geq y|J|.
\label{eq:hjz-once-loss}
\end{equation}
Moreover, the restrictedness of $J$ and the lower bound on $|L|$
show that $L$ is $y^3$-restricted.

Apply Lemma~\ref{lem:hjz-modified-first} to $L$.  Its first outcome is
excluded.  Its second outcome has order at least
$y^A|L|\geq y^{A+2}|J|$.  Its third gives
\[
                  (x^9|L|)^a\geq(x^{10}|J|)^a.
\]
In its fourth outcome, $k\geq y^{-g}$, so
$y\geq k^{-1/g}$, and
\[
 \frac{|L|}{k^{U_0}}
      \geq \frac{y|J|}{k^{U_0}}
      \geq \frac{|J|}{k^{U_0+1/g}}
      =\frac{|J|}{k^{U_1}}.
\]
Likewise, in its fifth outcome, $k\geq y^{-p}$ gives
\[
 \frac{|L|}{k^{P_0}}
      \geq \frac{|J|}{k^{P_0+1/p}}
      =\frac{|J|}{k^{P_1}}.
\]
This proves all five alternatives.
\end{proof}

\begin{lemma}[Iterative sparsification]
\label{lem:hjz-modified-iterative}
Let $0<x\leq\zeta^{10}$, and let $J$ be a
$\zeta^{10}$-restricted $\overline{\mathcal F}$-free graph.  At least
one of the following holds:
\begin{enumerate}
\item $J$ has an $x$-restricted induced subgraph of order at least
      $x^{22A}|J|$;
\item $\operatorname{hom}(J)\geq(x^{30A}|J|)^a$;
\item $J$ has a complete or anticomplete
\[
                    (k,|J|/k^{U_2})\text{-blockade},
                    \qquad k\geq2;
\]
\item $J$ has an $x$-sparse or pure
\[
                    (t,|J|/t^{P_*})\text{-blockade},
                    \qquad t\in[K_*,x^{-2}].
\]
\end{enumerate}
\end{lemma}

\begin{proof}
Put
\[
                              s=10(A+2).
\]
Suppose that none of the four outcomes holds.  We first prove that, for
every $y\in[x,\zeta^3]$ and every $y^{10/3}$-restricted induced
subgraph $L$ satisfying
\begin{equation}
                              |L|\geq y^s|J|,
\label{eq:hjz-iterative-L}
\end{equation}
there is a $y^{11/3}$-restricted induced subgraph of $L$ of order at
least $y^{A+2}|L|$.

Indeed, $y^{10/3}\leq\zeta y^3$, so
Lemma~\ref{lem:hjz-modified-once} applies to $L$.  Its first outcome,
with $t=y^{-1}$, has width at least
\[
 y^{A+2}|L|
   \geq y^{11(A+2)}|J|
   \geq \frac{|J|}{t^{29A}},
\]
because $11(A+2)\leq29A$.  Moreover
$t\in[\zeta^{-1},x^{-2}]$.  This is outcome~(4), since a complete
blockade is pure.

The second outcome gives the asserted restricted subgraph, because
$y^{1/3}\leq\zeta<1/2$ and hence
$2y^4\leq y^{11/3}$.  The third outcome implies outcome~(2), since
\[
             x^{10}|L|\geq x^{10+s}|J|
                       \geq x^{30A}|J|.
\]
In the fourth outcome, $k\geq y^{-g}>1$.  If $k\geq2$, then
\[
 \frac{|L|}{k^{U_1}}
      \geq \frac{y^s|J|}{k^{U_1}}
      \geq \frac{|J|}{k^{U_1+s/g}}
      \geq \frac{|J|}{k^{U_2}},
\]
where the last inequality follows from
$U_1+s/g=b+(10A+27)/g\leq b+27A/g=U_2$, since $A\geq4$.
Thus outcome~(3) holds.  Suppose instead that $1<k<2$.  The actual
number of blocks is an integer at least $k$, and hence is at least two.
Moreover, $k\geq y^{-g}$ and $k<2$ imply
\[
                         y^s>2^{-s/g}.
\]
Retaining two blocks, each retained block has order at least
\[
 \frac{|L|}{k^{U_1}}
   \geq \frac{y^s|J|}{k^{U_1}}
   > \frac{2^{-s/g}|J|}{2^{U_1}}
   = \frac{|J|}{2^{U_1+s/g}}
   \geq \frac{|J|}{2^{U_2}}.
\]
Thus outcome~(3) again holds, now with parameter $2$.

Finally, the fifth outcome satisfies
\[
 \frac{|L|}{k^{P_1}}
      \geq \frac{|J|}{k^{P_1+s/p}}
      =\frac{|J|}{k^{P_2}},
\]
and
\[
              k\in[\zeta^{-3p},x^{-2}].
\]
It therefore gives outcome~(4).  This proves the asserted
restricted-subgraph improvement.

Apply Lemma~\ref{lem:hjz-restricted-iteration} with
\[
 c=\zeta^{10},\qquad x_0=x^{10/3},\qquad
 b_1=\frac{11}{10},\qquad
 b_2=3(A+2),\qquad b_3=\frac{3(A+2)}{10}.
\]
Here $b_1b_2=b_2+b_3$.  The graph $J$ itself supplies the initial
$c$-restricted graph, and the preceding paragraph supplies the
inductive improvement after writing $\lambda=y^{10/3}$.  We obtain an
$x^{10/3}$-restricted induced subgraph of order at least
\[
 (x^{10/3})^{33(A+2)/10}|J|
       =x^{11(A+2)}|J|
       \geq x^{22A}|J|.
\]
Since $x^{10/3}\leq x$, this is outcome~(1), a contradiction.
\end{proof}

We next remove the fixed restrictedness hypothesis.  Choose any
$H_0\in\overline{\mathcal F}$.  By R\"odl's theorem
\cite{Rodl1986}, there is $\delta>0$ such that every
$\overline{\mathcal F}$-free graph contains a
$\zeta^{10}$-restricted induced subgraph of order at least
$\delta$ times its own order.  Choose an integer $D$ so large that
\begin{align}
 D&\geq\max\{8,U_2,2P_*\},&
 2^{-D}&<\zeta^{10},&
 2^{-D}&\leq\delta,                                      \label{eq:hjz-D-1}\\
 2^{-(D-U_2)}&\leq\delta,&
 2^{-(D-2P_*)}&\leq\delta\,3^{-P_*}.                     \label{eq:hjz-D-2}
\end{align}

\begin{lemma}[R\"odl reduction]
\label{lem:hjz-modified-rodl}
Let $0<x<2^{-D}$, and let $J$ be an
$\overline{\mathcal F}$-free graph with $|J|\geq x^{-D}$.  At least
one of the following holds:
\begin{enumerate}
\item $J$ has an $x$-restricted induced subgraph of order at least
      $x^{23A}|J|$;
\item $J$ has a pure or $x$-sparse
\[
                   (k,|J|/k^D)\text{-blockade},
                   \qquad k\in[2,x^{-1}];
\]
\item $\operatorname{hom}(J)\geq(x^{31A}|J|)^a$;
\item $J$ has a complete or anticomplete
\[
                   (k,|J|/k^D)\text{-blockade},
                   \qquad k\geq x^{-1}.
\]
\end{enumerate}
\end{lemma}

\begin{proof}
Choose a $\zeta^{10}$-restricted induced subgraph
$L\subseteq J$ with $|L|\geq\delta|J|$, and apply
Lemma~\ref{lem:hjz-modified-iterative} to $L$.

Its first outcome has order at least
\[
 \delta x^{22A}|J|
      \geq x^{22A+1}|J|
      \geq x^{23A}|J|,
\]
and its second outcome gives
\[
 (\delta x^{30A}|J|)^a
      \geq(x^{31A}|J|)^a.
\]
Here we used $\delta\geq2^{-D}>x$ and $A\geq1$.

Suppose that its third outcome occurs.  By
\eqref{eq:hjz-D-2}, for every $k\geq2$,
\[
 \frac{\delta|J|}{k^{U_2}}\geq\frac{|J|}{k^D}.
\]
If $k\leq x^{-1}$ this gives outcome~(2), since a complete or
anticomplete blockade is pure; if $k\geq x^{-1}$ it gives
outcome~(4).

It remains to consider its fourth outcome.  Let
\[
                              r=\lfloor\sqrt t\rfloor.
\]
By \eqref{eq:hjz-kstar}, $r\geq4$, and $r\leq x^{-1}$.  Retain any
$r$ blocks.  Since
\[
                         t<(r+1)^2\leq3r^2,
\]
their width is at least
\[
 \frac{\delta|J|}{t^{P_*}}
       \geq\frac{\delta3^{-P_*}|J|}{r^{2P_*}}
       \geq\frac{|J|}{r^D},
\]
where the last inequality follows from
\eqref{eq:hjz-D-2} and $r\geq2$.  Truncation preserves purity and
$x$-sparsity, so outcome~(2) follows.
\end{proof}

\begin{lemma}[Quantitative comb reduction]
\label{lem:hjz-poly-length}
If $\mathcal F$ is leaf-reducible and satisfies
Definition~\ref{def:comb-hypothesis}, then $\mathcal F$ is
generalized nice.  Consequently, if $\mathcal F$ is also wonderful,
then $\mathcal F$ has the Erd\H{o}s--Hajnal property.
\end{lemma}

\begin{proof}
We verify generalized niceness with
\begin{align*}
 C_1&=10D^2,& C_2&=D,& C_3&=156AD,& C_4&=a,\\
 C_5&=2D,& C_6&=5D,& C_7&=5D,& C_8&=116AD.
\end{align*}
These constants satisfy the numerical lower bounds in the definition
of generalized niceness.

Let $0<\varepsilon<1/2$, let $J$ be an
$\overline{\mathcal F}$-free graph, and put
\[
                              x=\varepsilon^{5D}.
\]
Then $x<2^{-D}$.  If $|J|\leq\varepsilon^{-1}$, then
\[
       \operatorname{hom}(J)\geq1
          \geq(\varepsilon^{C_3}|J|)^{C_4},
\]
so (GN2) holds.  If
\[
       \varepsilon^{-1}\leq|J|
             \leq\varepsilon^{-10D^2},
\]
choose $\lceil\varepsilon^{-1}\rceil$ singleton blocks.  Their width is
at least $\varepsilon^{10D^2}|J|$, and every pair is either complete
or weakly $\varepsilon^D$-sparse.  Thus (GN1) holds.

We may therefore assume that
\begin{equation}
                             |J|\geq\varepsilon^{-10D^2}.
\label{eq:hjz-large-J}
\end{equation}
Let $L$ be any induced subgraph with
$|L|\geq\varepsilon^D|J|$.  From
\eqref{eq:hjz-large-J},
\[
 |L|\geq\varepsilon^{-10D^2+D}
          \geq\varepsilon^{-5D^2}=x^{-D},
\]
so Lemma~\ref{lem:hjz-modified-rodl} applies to $L$.

If its first outcome occurs, then $J$ has an
$\varepsilon^{5D}$-restricted induced subgraph of order at least
\[
 x^{23A}|L|
    \geq\varepsilon^{115AD+D}|J|
    \geq\varepsilon^{116AD}|J|,
\]
which is (GN4).  Its third outcome gives
\[
 \operatorname{hom}(J)
      \geq(x^{31A}|L|)^a
      \geq(\varepsilon^{156AD}|J|)^a,
\]
which is (GN2).  In its fourth outcome,
$k\geq x^{-1}=\varepsilon^{-5D}$, and
\[
 \frac{|L|}{k^D}
       \geq\frac{\varepsilon^D|J|}{k^D}
       \geq\frac{|J|}{k^{2D}},
\]
because
$k^{-D}\leq\varepsilon^{5D^2}\leq\varepsilon^D$.
This is (GN3).

Consequently, unless one of (GN2)--(GN4) already holds, every induced
subgraph $L$ of order at least $\varepsilon^D|J|$ has a pure or
$x$-sparse
\[
                 (k,|L|/k^D)\text{-blockade},
                 \qquad k\in[2,x^{-1}].
\]
Lemma~\ref{lem:hjz-pure-sparse-conversion} now gives an
\[
  (\varepsilon^{-1},x^{2D}|J|)
   =(\varepsilon^{-1},\varepsilon^{10D^2}|J|)
\]
blockade whose blocks are pairwise complete or weakly
$\varepsilon^D$-sparse.  This is (GN1), and generalized niceness is
proved.

If $\mathcal F$ is also wonderful, then
Theorem~\ref{thm:hjz-generalized-nice} gives the
Erd\H{o}s--Hajnal property.
\end{proof}

This section is applied only to the graph $F$.  It does not prove the
$E$-graph or Bird cases, which remain the results of Huang, Ju, and
Zhou~\cite{HuangJuZhou2026}.

\section{Handle selection and the comb conclusion}
\label{sec:main-comb}

Retain the notation $T=\overline F$ from
Section~\ref{sec:target-local}.
We work in a $T$-free graph containing a comb with handles $h_i$ and
teeth $B_i$, together with a vertex $r$ complete to every tooth and
anticomplete to every handle.  By definition, $h_i$ is complete to $B_i$
and anticomplete to every other tooth.

We first establish two induced-copy lemmas for configurations in the
auxiliary graph on the handles.

\subsection{Induced-copy lemmas for handle configurations}

\begin{lemma}[Diamond tooth cut]
\label{lem:pivot-diamond-cut}
After relabelling four comb indices as $0,1,2,3$, suppose that
$h_0,h_1,h_2,h_3$ induce a diamond, with centres
$h_0,h_1$ and nonadjacent tips $h_2,h_3$.  Then
\[
                   B_0\cup B_1
       \quad\hbox{is anticomplete to}\quad
                   B_2\cup B_3.
\]
\end{lemma}

\begin{proof}
By symmetry it is enough to exclude an edge $z_0z_2$ with
$z_0\in B_0,z_2\in B_2$.  If this edge exists, then
\begin{equation}
 (a,b,c,d,e,f)\longmapsto(h_2,z_0,h_1,z_2,h_0,h_3)
 \label{eq:pivot-diamond-witness}
\end{equation}
sends the eight edges of $T$ to
\[
 h_2h_1,h_2z_2,h_2h_0,z_0z_2,z_0h_0,
 h_1h_0,h_1h_3,h_0h_3,
\]
and its seven nonedges to
\[
 h_2z_0,h_2h_3,z_0h_1,z_0h_3,h_1z_2,z_2h_0,z_2h_3.
\]
All signs follow from the comb incidences and the unique missing edge
of the diamond.  Thus \eqref{eq:pivot-diamond-witness} is an induced
$T$, a contradiction.
\end{proof}

\begin{lemma}[Clique-handle propagation]
\label{lem:pivot-clique-propagation}
If $h_ih_j$ and $xy$ are edges, where $x\in B_i,y\in B_j$, then for
every $z$ in a third tooth,
\begin{equation}
                         xz\in E(G)\quad\Longleftrightarrow\quad
                         yz\in E(G).
 \label{eq:pivot-clique-propagation}
\end{equation}
\end{lemma}

\begin{proof}
If $xz$ is an edge and $yz$ is not, then
\[
 (a,b,c,d,e,f)\longmapsto(y,h_i,r,h_j,x,z)
\]
sends the edge list of $T$ to
\[
 yr,yh_j,yx,h_ih_j,h_ix,rx,rz,xz
\]
and its nonedge list to
\[
 yh_i,yz,h_ir,h_iz,rh_j,h_jx,h_jz.
\]
The reverse implication follows by interchanging $i,j$.  Hence either
failure of \eqref{eq:pivot-clique-propagation} gives an induced $T$.
\end{proof}

\subsection{Selection in the auxiliary graph on the handles}

For a graph $Q$ and $x\in V(Q)$, write $N_Q[x]$ for its closed
neighbourhood in $Q$.

\begin{lemma}[Pivot--true-twin extraction]
\label{lem:pivot-twin-extraction}
Let $H$ be the graph induced by $\ell$ handles, where $\ell\geq16$.
At least one of the following holds.
\begin{enumerate}
 \item $H$ has a clique of order at least $\ell^{1/4}$;
 \item there is a set $I$ of at least $\ell^{1/4}/2$ handle indices
       such that, whenever $h_ih_j$ is an edge with $i,j\in I$, the
       whole interface $B_i$--$B_j$ is anticomplete.
\end{enumerate}
\end{lemma}

\begin{proof}
If $\Delta(H)<\sqrt\ell$, greedy stable-set selection gives a stable
set of order at least
\[
       \frac{\ell}{\Delta(H)+1}>\frac{\ell}{\sqrt\ell+1}
                                  \geq \frac{\sqrt\ell}{2}
                                  \geq \frac{\ell^{1/4}}{2}.
\]
It satisfies item~2 automatically.

We may therefore choose a handle $h_0$ with
$m=|N_H(h_0)|\geq\sqrt\ell$.  Put $Q=H[N_H(h_0)]$, and partition
$V(Q)$ by the equivalence relation
\[
                x\sim y\quad\Longleftrightarrow\quad
                N_Q[x]=N_Q[y].
\]
Every equivalence class is a clique: if $x\ne y$ have equal closed
neighbourhoods, then $x\in N_Q[y]$.  If one class has at least
$\sqrt m$ vertices, it is a clique in $H$ of order at least
$\sqrt m\geq\ell^{1/4}$, and item 1 holds.

Otherwise there are more than $\sqrt m$ classes.  Choose one handle
representative from every class, and let $I$ be their original comb
index set.
Then
\[
                       |I|>\sqrt m\geq\ell^{1/4}.
\]
Let $u,v\in I$ and suppose $h_uh_v\in E(Q)$.  Since these handles
represent different classes,
$N_Q[h_u]\ne N_Q[h_v]$.  Both closed neighbourhoods contain
$h_u,h_v$, so there is a handle $h_z\in
V(Q)\setminus\{h_u,h_v\}$ adjacent to exactly one of them; say
$h_zh_u$ is an edge and $h_zh_v$ is a nonedge.  The four handles
\[
                         h_0,h_u,h_v,h_z
\]
induce a diamond with centres $h_0,h_u$ and tips $h_v,h_z$.
Lemma~\ref{lem:pivot-diamond-cut} therefore makes $B_u$ anticomplete
to $B_v$.  This holds for every adjacent pair of representatives, so
item 2 follows.
\end{proof}

In item~2 of Lemma~\ref{lem:pivot-twin-extraction}, handle nonedges give
affine interfaces and handle edges give constant-zero interfaces.

\subsection{Components, anticomponents, and additive separation}

We use the following component--anticomponent extraction, with explicit
width bounds.

\begin{lemma}[Components and anticomponents]
\label{lem:pivot-carriers}
Let teeth $B_i$, $i\in I$, have order at least $w$, and let $L\geq2$
be an integer.
One of the following holds.
\begin{enumerate}
 \item $G$ has a complete or anticomplete blockade of length $L$ and
       width at least $w/L^3$;
 \item for every $i\in I$ there are
       $X_i\subseteq C_i\subseteq B_i$ such that $G[C_i]$ is connected,
       $G[X_i]$ is anticonnected, and
       \[
                         |X_i|\geq w/L^2.
       \]
\end{enumerate}
If, in addition, every handle edge on $I$ has an anticomplete tooth
interface, then in the second outcome every $X_i$--$X_j$ matrix is additively separable over $\mathbb F_2$:
\begin{equation}
 M_{ij}(x,y)=\lambda_i^j(x)\oplus\lambda_j^i(y)
                                      \oplus\gamma_{ij}.
 \label{eq:pivot-carrier-affine}
\end{equation}
\end{lemma}

\begin{proof}
Partition a tooth into its connected components.  If none has order
at least $w/L$, greedily group whole components into $L$ unions, each
of order at least $w/L^2$.  Indeed, a completed union has order less
than $w/L^2+w/L$, and after $L-1$ such unions the unused total is still
greater than $w/L^2$.  The unions form an anticomplete blockade.
Otherwise retain a connected component $C_i$ of order at least $w/L$.

Inside $C_i$, apply the complementary argument to its anticomponents.
It either gives a complete $L$-blockade of width at least
$|C_i|/L^2\geq w/L^3$, or gives an anticomponent
$X_i$ of order at least $|C_i|/L\geq w/L^2$.  Carry this out in every
tooth unless an immediate blockade has already occurred.

It remains to prove \eqref{eq:pivot-carrier-affine}.  If $h_ih_j$ is an edge, the additional
hypothesis says that the interface is constant zero, which has the
form \eqref{eq:pivot-carrier-affine}.  Suppose $h_ih_j$ is a nonedge.  If $x,x'\in X_i$ are
nonadjacent and $y,y'\in C_j$ are adjacent, the rectangle
$\{x,x'\}\times\{y,y'\}$ cannot contain exactly one edge: after
relabelling, its witness is
\[
 (a,b,c,d,e,f)\longmapsto(x,x',y,h_i,r,y').
\]
It cannot contain exactly three edges either: if $x'y'$ is its unique
nonedge, a witness is
\begin{equation}
 (a,b,c,d,e,f)\longmapsto(x,x',y',h_i,y,h_j).
 \label{eq:pivot-carrier-three-edge-witness}
\end{equation}
The last nonedge required in
\eqref{eq:pivot-carrier-three-edge-witness} is precisely $h_ih_j$.
Thus every
such rectangle has even parity.  Join two rows by a path in
$\overline{G[X_i]}$ and two columns by a path in $G[C_j]$, and sum the
elementary parity equalities.  Every rectangle of the full
$X_i$--$C_j$ matrix has even parity.  Fixing one row and one column
gives the displayed additively separable representation; restricting the
columns from $C_j$ to $X_j$ proves \eqref{eq:pivot-carrier-affine}.
\end{proof}

The results of Sections~\ref{sec:all-mixed} and~\ref{sec:outer-graph}
give the following extraction.  Its hypotheses state explicitly where
handle nonadjacency is used.

\begin{corollary}[Extraction from additively separable interfaces]
\label{cor:pivot-affine-extraction}
There are constants $\rho,\sigma>0$ and an integer $m_0$ with the
following property.  In a $T$-free graph containing a comb and a vertex
complete to all teeth and anticomplete to all handles, let
$X_1,\ldots,X_m$ be anticonnected subsets of distinct
teeth, where $m\geq m_0$ and $|X_i|\geq W$ for every $i$.  Suppose
every pairwise interface has the form
\[
 M_{ij}(x,y)=\lambda_i^j(x)\oplus\lambda_j^i(y)
                         \oplus\gamma_{ij},
\]
and suppose $h_ih_j$ is a nonedge whenever the $X_i$--$X_j$ interface
is mixed.  Then one of the following holds:
\begin{enumerate}
 \item a pure blockade of length at least $m^\rho$ and width at least
       $W$;
 \item a complete or anticomplete blockade of length at least
       $m^\sigma$ and width at least $W/m$.
\end{enumerate}
In either outcome, the blocks lie in pairwise distinct sets $X_i$; in
particular, the actual length of the blockade is at most $m$.
\end{corollary}

\begin{proof}
This is Theorem~\ref{thm:f124-affine-interface-extraction}.  Its
hypotheses are exactly the hypotheses displayed above.
\end{proof}

\subsection{Clique handles}

We also use the following aggregation statement.

\begin{lemma}[Clique-handle extraction]
\label{lem:pivot-clique-extraction}
For $m\geq144$ clique handles with teeth of order at least $w$, there
is a complete or anticomplete blockade of length at least $\sqrt m/2$
and width at least $w/(2m)$.  For every $m\geq3$ there is a
pure pair of width $w/(4m)$.
\end{lemma}

\begin{proof}
First trim every tooth to the common integer order
$W=\lceil w\rceil$.  This is possible because an integer tooth order
at least $w$ is at least $\lceil w\rceil$, and $W\geq w$.  Discard the
edges internal to the teeth, and let
$R$ consist of the cross-tooth edges on the trimmed teeth.
Lemma~\ref{lem:pivot-clique-propagation} says that, for vertices
in three different teeth,
\begin{equation}
                  xy,yz\in E(R)\quad\Longrightarrow\quad xz\in E(R).
 \label{eq:clique-cross-transitivity}
\end{equation}
Let $Q$ be the complementary cross graph: vertices in different teeth
are adjacent in $Q$ exactly when they are nonadjacent in $R$.
Condition~\eqref{eq:clique-cross-transitivity} is precisely the
hypothesis of the no-two-edge aggregation lemma,
Lemma~\ref{lem:f124-no-two-edge-aggregation}.  It gives a pure
pair in $Q$ of width at least $W/(4m)\geq w/(4m)$, and, for $m\geq144$,
a complete or anticomplete blockade in $Q$ of length at least
$\sqrt m/2$ and width at least $W/(2m)\geq w/(2m)$.  Complementing the
cross signs turns either conclusion into the stated conclusion in the
host graph $G$.
\end{proof}

\subsection{Verification of the quantitative comb condition}

\begin{theorem}
\label{thm:F-comb-conclusion}
The singleton family $\{F\}$ satisfies the quantitative comb condition of
Definition~\ref{def:comb-hypothesis}.
\end{theorem}

\begin{proof}
Let $\rho,\sigma>0$ and $m_0$ be the constants in
Corollary~\ref{cor:pivot-affine-extraction}.  Decreasing $\rho$ if
necessary, assume that $0<\rho\leq1$.  Set
\[
 q_0=\frac{1}{8},\qquad p_0=\frac{\rho}{8},\qquad
 c_*=\min\left\{\frac{1}{16},\frac{\sigma}{8}\right\}.
\]
Choose an integer $\ell_0$ so large that, whenever $\ell\geq\ell_0$,
all of the following hold:
\begin{enumerate}
 \item $L=\lfloor\ell^{q_0}\rfloor\geq2$ and
       $L\geq\ell^{1/16}$;
 \item $\ell^{1/4}/2\geq m_0$;
 \item every clique of at least $\ell^{1/4}$ handles has order at
       least $144$;
 \item
 \[
  \frac{1}{2}\ell^{1/8}\geq\ell^{c_*},\qquad
  2^{-\rho}\ell^{\rho/4}\geq\ell^{p_0},\qquad
  2^{-\sigma}\ell^{\sigma/4}\geq\ell^{c_*};
  \]
\end{enumerate}

After fixing $\ell_0$, put
\[
 c=\min\left\{c_*,\frac{\log 2}{\log\ell_0}\right\},
 \qquad p=p_0,\qquad q=\frac{1}{4},
\]
and choose
\begin{equation}
 b\geq\max\left\{3,\frac{2}{c},
          \log_2\bigl(\max\{8,4\ell_0\}\bigr)\right\}.
 \label{eq:comb-b-choice}
\end{equation}
We verify Definition~\ref{def:comb-hypothesis} with
$a=1$, $g=c$, and these fixed $b,p,q$.  Outcome~(H) will not be needed.

Consider first $\ell\geq\ell_0$.  Apply
Lemma~\ref{lem:pivot-twin-extraction}.  Suppose it gives a clique of
$m\geq\ell^{1/4}$ handles.  By
Lemma~\ref{lem:pivot-clique-extraction}, there is a complete or
anticomplete blockade of length at least
$\sqrt m/2\geq\ell^{1/8}/2$ and width at least
\[
                         \frac{w}{2m}\geq \frac{w}{2\ell}.
\]
Keep exactly $k=\lceil\ell^c\rceil$ of its blocks.  By the choice of
$\ell_0$, its integer length is at least
$\lceil\ell^{c_*}\rceil\geq k$, so this is possible; also
$k\geq\ell^c$.  Since
$b c\geq2$, we have $k^b\geq\ell^{bc}\geq\ell^2\geq2\ell$, and
therefore the retained width is at least $w/k^b$.  This is outcome (U).

We may thus assume that Lemma~\ref{lem:pivot-twin-extraction} gives a
set $I$ satisfying
\[
                         |I|\geq\ell^{1/4}/2,
\]
and such that every handle edge on $I$ has an anticomplete tooth
interface.  Apply Lemma~\ref{lem:pivot-carriers} with
$L=\lfloor\ell^{1/8}\rfloor$.  If its first conclusion holds, it
gives a complete or anticomplete blockade with $k=L\geq\ell^c$ blocks
and width at least $w/L^3\geq w/k^b$, because $b\geq3$.  This is again
outcome (U).

In the other conclusion, there are
$X_i\subseteq C_i\subseteq B_i$, $i\in I$, such that $G[C_i]$ is
connected, $G[X_i]$ is anticonnected, and
\[
                         |X_i|\geq \frac{w}{L^2}
                                  \geq \frac{w}{\ell^{1/4}}.
\]
Every interface $X_i$--$X_j$ is additively separable over
$\mathbb F_2$.  On a handle edge the interface is constant zero; hence
every mixed interface lies over a handle nonedge.  We may therefore
apply Corollary~\ref{cor:pivot-affine-extraction} with
$m=|I|$ and $W=w/\ell^{1/4}$.

In its pure conclusion, the length $k$ satisfies
\[
 k\geq |I|^\rho
       \geq2^{-\rho}\ell^{\rho/4}
       \geq\ell^p,
 \qquad k\leq |I|\leq\ell,
\]
and the width is at least $w/\ell^{1/4}=w/\ell^q$.  This is outcome
(P).

In its complete-or-anticomplete conclusion, the length is at least
\[
 |I|^\sigma\geq2^{-\sigma}\ell^{\sigma/4}
                   \geq\ell^{c_*}\geq\ell^c,
\]
and the width is at least
\[
 \frac{W}{|I|}\geq \frac{w}{\ell^{5/4}}.
\]
Keep $k=\lceil\ell^c\rceil$ blocks.  These blocks are available because
the original integer length is at least
$\lceil\ell^{c_*}\rceil\geq k$.  Also $k^b\geq\ell^2$, so
$w/\ell^{5/4}\geq w/k^b$.  This is outcome (U).  This proves the
required conclusion for $\ell\geq\ell_0$.

It remains to consider $4\leq\ell<\ell_0$.  If two handles are
nonadjacent, apply Lemma~\ref{lem:pivot-carriers} with $L=2$ to their
two teeth.  Its first conclusion gives a complete or anticomplete pair
of width at least $w/8$.  In its second conclusion the two retained
sets have order at least $w/4$, their interface is additively separable,
and a largest fibre of each of the two binary endpoint functions gives
a complete or anticomplete pair of width at least $w/8$.  If every pair
of handles is adjacent, Lemma~\ref{lem:pivot-clique-extraction} gives a
complete or anticomplete pair of width at least $w/(4\ell_0)$.

In every bounded case take $k=2$.  The definition of $c$ gives
$\ell^c\leq2$, while the choice of $b$ in
\eqref{eq:comb-b-choice} gives
$2^b\geq\max\{8,4\ell_0\}$.  Thus the pair has width at least
$w/2^b$, and outcome (U) holds.  The theorem follows.
\end{proof}

\begin{proof}[Proof of Theorem~\ref{thm:main}]
Proposition~\ref{prop:f124-leaf-pair} shows that $\{F\}$ is
leaf-reducible, and Proposition~\ref{prop:f124-wonderful} shows that it
is wonderful.  Theorem~\ref{thm:F-comb-conclusion} and
Lemma~\ref{lem:hjz-poly-length} imply that $\{F\}$ is generalized
nice.  Theorem~\ref{thm:hjz-generalized-nice} now gives the
Erd\H{o}s--Hajnal property for $F$.
\end{proof}

\section*{Acknowledgements}
\addcontentsline{toc}{section}{Acknowledgements}

We used GPT-5.6 Sol and an agentic harness built around GPT-5.6 Sol and Claude Fable 5 to assist with literature searches, hypothesis testing, the exploration and elimination of potential approaches, wording refinement, and manuscript proofreading. Viet-Hoang Tran thanks Tung Nguyen for introducing him to the background and literature of the Erdős–Hajnal conjecture. He also thanks Tho Tran Huu, Khoi M. N. Nguyen, and Hieu M. Vu for assistance with hardware-related matters. We are very grateful to Dung V. Nguyen and Quang X. Nguyen for their technical support in the use of the AI tools and agentic system.

\bibliographystyle{plain}
\phantomsection
\addcontentsline{toc}{section}{References}
\begingroup
\small
\raggedright
\bibliography{example_paper}

@article{ErdosHajnal1989,
  author  = {Paul Erd{\H{o}}s and Andr{\'a}s Hajnal},
  title   = {Ramsey-Type Theorems},
  journal = {Discrete Applied Mathematics},
  volume  = {25},
  number  = {1--2},
  pages   = {37--52},
  year    = {1989},
  doi     = {10.1016/0166-218X(89)90045-0}
}

@article{Rodl1986,
  author  = {Vojt{\v{e}}ch R{\"o}dl},
  title   = {On Universality of Graphs with Uniformly Distributed Edges},
  journal = {Discrete Mathematics},
  volume  = {59},
  number  = {1--2},
  pages   = {125--134},
  year    = {1986},
  doi     = {10.1016/0012-365X(86)90076-2}
}

@article{AlonPachSolymosi2001,
  author  = {Noga Alon and J{\'a}nos Pach and J{\'o}zsef Solymosi},
  title   = {Ramsey-Type Theorems with Forbidden Subgraphs},
  journal = {Combinatorica},
  volume  = {21},
  number  = {2},
  pages   = {155--170},
  year    = {2001},
  doi     = {10.1007/s004930100016}
}

@article{ChudnovskySafra2008,
  author  = {Maria Chudnovsky and Shmuel Safra},
  title   = {The {Erd{\H{o}}s--Hajnal} Conjecture for Bull-Free Graphs},
  journal = {Journal of Combinatorial Theory, Series B},
  volume  = {98},
  number  = {6},
  pages   = {1301--1310},
  year    = {2008},
  doi     = {10.1016/j.jctb.2008.02.005}
}

@article{ChudnovskyScottSeymourSpirkl2023,
  author  = {Maria Chudnovsky and Alex Scott and Paul Seymour and Sophie Spirkl},
  title   = {{Erd{\H{o}}s--Hajnal} for Graphs with No 5-Hole},
  journal = {Proceedings of the London Mathematical Society},
  volume  = {126},
  number  = {3},
  pages   = {997--1014},
  year    = {2023},
  doi     = {10.1112/plms.12504},
  eprint  = {2102.04994},
  archivePrefix = {arXiv}
}

@article{BucicNguyenScottSeymour2024,
  author  = {Matija Buci{\'c} and Tung Nguyen and Alex Scott and Paul Seymour},
  title   = {Induced Subgraph Density. {I}. {A} Loglog Step towards {Erd{\H{o}}s--Hajnal}},
  journal = {International Mathematics Research Notices},
  volume  = {2024},
  number  = {12},
  pages   = {9991--10004},
  year    = {2024},
  doi     = {10.1093/imrn/rnae065},
  eprint  = {2301.10147},
  archivePrefix = {arXiv},
  primaryClass = {math.CO}
}

@article{NguyenScottSeymourDensityIV,
  author  = {Tung Nguyen and Alex Scott and Paul Seymour},
  title   = {Induced Subgraph Density. {IV}. {New} Graphs with the {Erd{\H{o}}s--Hajnal} Property},
  journal = {Transactions of the American Mathematical Society},
  year    = {2026},
  note    = {Accepted 7 April 2026; arXiv:2307.06455v4 [math.CO]},
  eprint  = {2307.06455},
  archivePrefix = {arXiv},
  primaryClass = {math.CO}
}

@article{NguyenScottSeymourP5,
  author  = {Tung Nguyen and Alex Scott and Paul Seymour},
  title   = {Induced Subgraph Density. {VII}. {The} Five-Vertex Path},
  journal = {Proceedings of the London Mathematical Society},
  volume  = {132},
  number  = {3},
  pages   = {e70133},
  year    = {2026},
  doi     = {10.1112/plms.70133},
  eprint  = {2312.15333},
  archivePrefix = {arXiv}
}

@article{BucicFoxPham2024,
  author  = {Matija Buci{\'c} and Jacob Fox and Huy Tuan Pham},
  title   = {Equivalence between {Erd{\H{o}}s--Hajnal} and Polynomial {R{\"o}dl} and {Nikiforov} Conjectures},
  journal = {Bulletin of the London Mathematical Society},
  year    = {forthcoming},
  note    = {arXiv:2403.08303v2 [math.CO]},
  eprint  = {2403.08303},
  archivePrefix = {arXiv},
  primaryClass = {math.CO}
}

@misc{HuangJuZhou2026,
  author       = {Shenwei Huang and Yiao Ju and Yidong Zhou},
  title        = {{Erd{\H{o}}s--Hajnal} beyond the Five-Vertex Path},
  year         = {2026},
  howpublished = {Preprint},
  note         = {arXiv:2606.06258v2 [math.CO], revised 8 June 2026},
  url          = {https://arxiv.org/abs/2606.06258}
}
\endgroup

\end{document}